\documentclass{article}

\usepackage{amsfonts}
\usepackage[utf8x]{inputenc}
\usepackage[T1]{fontenc}
\usepackage{lmodern}
\usepackage{amsmath,amsthm,amssymb,bm}
\usepackage[hidelinks]{hyperref}
\usepackage{cleveref,nameref}
\usepackage[a4paper]{geometry}
\usepackage{mathtools}
\usepackage{enumitem}

\newtheorem{theorem}{Theorem}[section]
\newtheorem{lemma}{Lemma}[section]
\newtheorem{proposition}[theorem]{Proposition}
\newtheorem{corollary}[theorem]{Corollary}
\theoremstyle{definition}
\newtheorem{definition}{Definition}[section]
\newtheorem{remark}{Remark}

\DeclareMathOperator{\diver}{div}
\DeclareMathOperator{\sign}{sign}

\newcommand{\modint}{\displaystyle\copy\tratto\kern-10.4pt\int\limits}

\def\R{\mathbb R}
\def\N{\mathbb N}

\def\OV{\overline}

\def\O{\Omega}
\def\a{\alpha}

\def\d{\delta}
\def\eps{\varepsilon}
\def\f{\varphi}
\def\g{\gamma}
\def\p{\partial}
\def\s{\sigma}

\def\LEQ{\leqslant}
\def\GEQ{\geqslant}

\def\WT{\widetilde}
\def\DST{\displaystyle}

\def\HF{\hfill{$\diamondsuit$\\ }}

\def\div{{\rm div\,}}
\def\measure{{\rm measure\,}}
\def\exp{{\rm exp\,}}

\def\O{\Omega}

\def\a{\alpha}

\def\d{\delta}
\def\f{\varphi}
\def\p{\partial}
\def\s{\sigma}

\def\LEQ{\leqslant}
\def\GEQ{\geqslant}

\def\Log{{\rm log\,}}

\def\DST{\displaystyle}

\def\calD{\mathcal D}

\def\g{u}

\begin{document}

\title{Existence and uniqueness of solutions of Schr\"{o}dinger type
stationary equations
with very singular potentials \\
without prescribing
boundary conditions \\ 
and some applications}

\author{
	J.I. D\'iaz%
	\thanks{Instituto de Matem\'atica Interdisciplinar \& Dpto. de Matem\'atica Aplicada,
		Universidad Complutense de Madrid,
		Plaza de las Ciencias, 3, 28040 Madrid, Spain. 
		J.I. D\'iaz: \url{jidiaz@ucm.es}, D. G\'omez-Castro: \url{dgcastro@ucm.es}}
		\and
		D. G\'omez-Castro\footnotemark[1]
		\and
		J.M. Rakotoson%
		\thanks{Universit\'e de Poitiers,
	Laboratoire de Math\'ematiques et Applications - UMR CNRS 7348 - SP2MI, France.	
	Bd Marie et Pierre Curie,  T\'el\'eport 2,  F-86962 Chasseneuil Futuroscope Cedex, France.
	\url{rako@math.univ-poitiers.fr}} 
		\thanks{corresponding author}
	}

\maketitle

\centerline{\it  Dedicated to the Memory of Robert JANIN}
\centerline{\it Related to the  International Conference in Analysis  POITIERS, March 29-30, 2017}

\begin{abstract}
	Motivated mainly by the localization over an open bounded set $\Omega$ of $\mathbb R^n$ of solutions of the Schrödinger equations, we consider the Schrödinger equation over $\Omega$ with a very singular potential $V(x) \ge C d (x, \partial \Omega)^{-r}$ with $r\ge 2$ and a convective flow $\vec U$. We prove the existence and uniqueness of a very weak solution of the equation, when the right hand side datum $f(x)$ is in $L^1 (\Omega, d(\cdot, \partial \Omega))$, even if no boundary condition is a priori prescribed. We prove that, in fact, the solution necessarily satisfies (in a suitable way) the Dirichlet condition $u = 0$ on $\partial \Omega$. These results improve some of the results of the previous paper by the authors in collaboration with Roger Temam. In addition, we prove some new results dealing with the $m$-accretivity in $L^1 (\Omega, d(\cdot, \partial \Omega)^ \alpha)$, where $\alpha \in [0,1]$, of the associated operator, the corresponding parabolic problem and the study of the complex evolution Schrödinger equation in $\mathbb R^n$. 
	
	
\end{abstract}

\section{Introduction}

The main goal of this paper is to improve some of the results of a previous
paper by the authors in collaboration with R. Temam \cite{DGRT}, as well as
some of the recent researches presented in \cite{OrPon}, concerning the Schr\"{o}%
dinger type stationary equations with a very singular potentials and/or a possibly unbounded convective flow
\begin{equation}
-\Delta u +\vec U(x)\cdot \nabla u +V(x)u =f(x)\text{ in }%
\Omega ,  \label{equation}
\end{equation}%
where $\Omega $ is an open subset of $\mathbb R^{n}$and $f\in
L^{1}(\Omega ,\delta )$, with 
\begin{equation}
\delta (x):=d(x,\partial \Omega ).
\end{equation}%
We assume given a convective flow $\vec U\in L^{n}(\Omega )^{n}$ \ such
that 
\begin{equation}
\begin{cases}
\diver\vec U=0 & \Omega , \\ 
\vec U\cdot \vec{\nu}=0 & \partial \Omega ,%
\end{cases}
\label{eq:incompressible strong}
\end{equation}%
with $\vec \nu$ the unit exterior normal vector to $\partial \Omega$ and a potential $V(x)$ in the general class of functions satisfying $V\in
L_{loc}^{1}(\Omega ),V\geq 0$ a.e. on $\Omega $. Our main motivation is to deal
with ``very singular potentials'' in the
sense that they satisfy
\begin{equation}
V(x)\geq \frac{C}{\delta (x)^{r}}\text{ for some }r\geq 2,\text{ near }%
\partial \Omega .  \label{singular potential}
\end{equation}%
but many results are obtained merely for $V \ge 0$ when $f$ behaves suitably near $\partial \Omega$.
We send the reader to \cite{DGRT} for considerations and references
concerning the case of \textquotedblleft moderate
singular\textquotedblright\ potentials corresponding to $r\in (0,2).$ Notice
that our purpose, as already indicated in the title of the paper, is to prove
the existence and uniqueness of a suitable class of solutions of \eqref{equation} without prescribing any boundary condition in an explicit way.
Nevertheless, we shall demand the solutions to have a certain integrability condition
which implicitly assumes some behaviour on $\partial \Omega $: we
shall enter into details later.

In our previous paper \cite{DGRT} we offered a set of relevant applications
leading to the consideration of problem (\ref{equation}). In the special
case of $\vec U=\vec{0}$ some of those motivations where: linearization of
singular and /or degenerate nonlinear equations, shape optimization in
Chemical Engineering and, very specially, the study of ground solutions $%
\bm{\psi }(t,x)=e^{-iEt}u (x)$ of the Schr\"{o}dinger equation%
\begin{equation}
\begin{dcases} i\frac{\partial \bm{\psi }}{\partial t}=-\Delta
\bm{\psi }+V(x)\bm{\psi } & \text{in }(0,\infty )\times \mathbb
R^{n} \\ \bm{\psi }(0,x)=\bm{\psi }_{0}(x) & \text{on } \mathbb
R^{n}\end{dcases}  \label{Schroedinger  whole space}
\end{equation}%
for very singular potentials (i.e., satisfying \eqref{singular potential}) which try to
confine the wave function $\bm \psi$ of the particle in the domain $\Omega $ of $\mathbb R^{n}.$ A very interesting source of concrete singular potentials
examples was described in the long paper \cite{Frank} where only asymptotic
technics were sketched for the treatment of the problems. We recall that the
confinement takes place once that we prove that the solutions of (\ref%
{equation}) are, in fact, \textquotedblleft flat
solutions\textquotedblright\ (in the sense that $u =\frac{\partial
u }{\partial n}=0$ on $\partial \Omega $).

Concerning the case $\vec U\neq \vec {0}$ the main motivation mentioned in 
\cite{DGRT} was the study of the vorticity equation in Fluid Mechanics.
Schrödinger equations involving also a flux term, motivated by some
questions in Control Theory, were already considered also by several authors
when proving the \textquotedblleft unique continuation
property\textquotedblright\ (see, e.g. \cite{Ionescu-Kening} and its
references). Notice that the existence of flat solutions to this equation
implies the failure of the \textquotedblleft unique continuation
property\textquotedblright\ \ for such very singular class of potentials.

So, roughly speaking, the aim of this paper is to study the problem 
\begin{subequations}
\label{eq:strong problem formulation}
\begin{align} 
	Au &=f \text{ in }\Omega ,
	\label{eq:strong problem formulation interior} \\
	u  &=0 \text{ on }\partial \Omega ,
\end{align}
where 
\end{subequations}
\begin{equation}
	Au =-\Delta u +\vec U\cdot \nabla u +Vu .
\end{equation}
The content of this paper is organized as follows: after a short presentation of notations, definitions and previous results (in \Cref{sec:2}), we list in \Cref{sec:3} some of the main new results in this paper. The equivalence between two different notions of very weak solutions of the equation under considerating is proved in \Cref{sec:4} by means of a sharper approximation argument applied to the test functions. \Cref{sec:5} contains the proof of the new existence and regularity regularity results, while the uniqueness of such solutions is considered in \Cref{sec:6}. Here the main tool is a new ``local type Kato inequality'' in which no use is made on possible boundary conditions (in the standard sense). The analysis of the solution when the right hand side datum $f$ is in $L^1( \Omega; \delta ^ \alpha)$ with $\alpha \in [0,1]$ is made in \Cref{sec:7}. Finally, \Cref{sec:8} collects several applications. In \Cref{subsec:8.1} we prove the $m$-accretiveness of the operator in $L^1 (\Omega, d(\cdot, \partial \Omega)^\alpha)$ (and in $L^p (\Omega, d(\cdot, \partial \Omega)^\alpha)$ when $\vec U = 0$ or $\alpha = 0$). Some consequences in terms of the associated parabolic problem are presented. \Cref{subsec:8.2} deals with the evolution (complex) Schrödinger problem in $\mathbb R^n$ associated to the very singular potential. We prove the localization of the solution in the sense that if $\textrm{supp}\,\psi_0 \subset \overline \Omega$ then $\textrm{supp}\, \psi (t, \cdot) \subset \overline \Omega$, for all $t \ge 0$.

\section{Notations, definitions and previous results} \label{sec:2}
We shall adopt the same notations as in our previous paper \cite{DGRT}.
We set 
$$
	L^0(\O)=\Big\{v:\O\to\R\hbox{ Lebesgue measurable}\Big\}
$$ 
and we denote    by $L^p(\O)$ the usual Lebesgue space $1\LEQ p\LEQ +\infty$.
Although it is not too often used, we shall use the notation 
$$W^{1,p}(\O)=W^1L^p(\O)$$
for the  associated Sobolev space. We need the following definitions:
\begin{definition}[of the distribution function and monotone rearrangement]
	Let $u\in L^0(\O)$. The distribution function of $u$ is the decreasing function
	\begin{eqnarray*}
		m=m_u:\R &\to &[0,|\O|] \\
			t & \mapsto & \measure\big\{ x:u(x)>t\big\}=|\big\{u>t\big\}|.
	\end{eqnarray*} 
	The generalized inverse $u_*$ of $m$ is defined by, for $s\in[0,|\O|[$,
	$$
		u_*(s)=\inf\Big\{t:|\big\{u>t\big\}|
	\LEQ s\Big\},
	$$
	and is called the decreasing rearrangement of $u$. We shall set $\O_*=]0,|\O|\,[.$
\end{definition}
\begin{definition}\label{d5}
	Let $1\LEQ p\LEQ +\infty,\ 0<q\LEQ+\infty$ :
	\begin{itemize}
		\item If $q<+\infty$, one defines the following norm for $u\in L^0(\O)$
	$$\|u\|_{p,q}=\|u\|_{L^{p,q}}:=\left[\int_{\O_*}\left[t^{\frac1p}|u|_{**}(t)\right]^q\frac{dt}t\right]^{\frac1q}\hbox{ where }|u|_{**}(t)=\dfrac1t\int_0^t|u|_*(\s)d\s.$$
		\item If $q=+\infty$,
	$$\|u\|_{p,\infty}=\sup_{0<t\LEQ|\O|}t^{\frac1p}|u|_{**}(t)
	.$$
	The space 
	\begin{equation} 
		L^{p,q}(\O)=\Big\{ u\in L^0(\O):\|u\|_{p,q}<+\infty\Big\}
	\end{equation}
	 is called a {\bf Lorentz space}.
	
		\item If $p=q=+\infty,$ $L^{\infty,\infty}(\O)=L^\infty(\O).$	
		\item The dual of $L^{1,1}(\O)$ is called $L_\exp(\O)$
	\end{itemize}
\end{definition}

\begin{remark}
We recall that $L^{p,q}(\O)\subset L^{p,p}(\O)=L^p(\O)$ for any $p>1,\ q\GEQ1$.
\end{remark}

\begin{definition}
If $X$  is a Banach space in $L^0(\O)$, we shall denote the Sobolev space associated to $X$ by
$$W^1X=\Big\{\f\in L^1(\O):\nabla\f\in X^n\Big\}$$ or more generally for $m\GEQ1$,
$$ W^mX=\Big\{\f\in W^1X,\ \forall\,\a=(\a_1,\ldots,\a_n)\in\N^n,\ |\a|=\a_1+\ldots+\a_n\LEQ m,\ D^{|\a|}\f\in X\Big\}.$$
We also set
$$W^1_0X=W^1X\cap W_0^{1,1}(\O).$$
\end{definition}
We shall often use the principal eigenvalue $\f_1 \in W_2$ of the homogeneous Dirichlet problem
\begin{equation}\label{eq03}
	\begin{dcases}
		-\Delta\f_1=\lambda_1\f_1 & \hbox{ in }\O,\\
		\f_1=0 & \hbox{ on }\p\O,
	\end{dcases}
\end{equation}
where
\begin{equation}
	W_{2}=\left\{ \varphi \in C^{2}(\bar{\Omega}):\varphi =0\text{ in }\partial
	\Omega \right\} .
\end{equation}%
We also need to recall the Hardy's inequality in $L^{n^{\prime },\infty }$ \
saying that 
\begin{equation}
\int_{\Omega }\frac{|u |}{\delta }\leq C\Vert \nabla u \Vert
_{L^{n^{\prime },\infty }}\qquad \forall u \in W_{0}^{1}L^{n^{\prime
},\infty }(\Omega ),
\end{equation}
with $n' = \frac n {n-1}$. This inequality can be obtained from the results of \cite{RakoJFA} (see also \cite{DR}) since $W_0^1 L^{n',\infty} (\Omega) \subset W_0^1 (\Omega; 1 + |\log \delta|)$.
\begin{definition} In the weak setting, by %
\eqref{eq:incompressible strong} we will mean 
\begin{equation}
\int_{\Omega }\varphi \nabla \phi \cdot \vec U=-\int_{\Omega }\phi \nabla
\varphi \cdot \vec U \qquad \forall \phi ,\varphi \in W_{2}.
\label{eq:incompressible u weak}
\end{equation}%
\end{definition}
In fact we will consider one of the following general assumptions (independently of the singularity of $V$): 
\begin{equation*}
\begin{cases}
V\in L_{loc}^{1}(\Omega ),V\geq 0, \\ 
\vec U\in L^{p,1}(\Omega )^n ,\text{ for some }p>n,\text{ and such that }%
\eqref{eq:incompressible u weak} \textrm{ holds}.%
\end{cases}
\label{eq:H with p greater n}
\tag{H$_1$}
\end{equation*}%
or 
\begin{equation*}
\begin{cases}
V\in L_{loc}^{1}(\Omega ),V\geq 0, \\ 
\vec U\in L^{n,1}(\Omega )^n,\text{ with small norm (as in Theorem 4.1 in \cite{DGRT}), and such
that }\eqref{eq:incompressible u weak} \textrm{ holds}.%
\end{cases}
\label{eq:H with p equal n}
\tag{H$_2$}
\end{equation*}%
Most frequently we will assume that 
\begin{equation}
\text{either \eqref{eq:H with p greater n} or \eqref{eq:H with p equal n} holds}.  \tag{H}  \label{H}
\end{equation}%
\begin{definition}
	Under assumption \eqref{H}, the local very weak formulation of 
	\eqref{eq:strong
		problem formulation interior} results
	\begin{equation}
	\int_{\Omega }u (-\Delta \phi -\vec U\cdot \nabla \phi
	+V\phi )=\int_{\Omega }f\phi \qquad \forall \varphi \in \mathcal{C}%
	_{c}^{2}(\Omega ).
	\end{equation}%
	For $V\in L^{1}(\Omega ,\delta )$, we say that $u $ is a ``very weak
	solution in the sense of Brezis'' of \eqref{eq:strong problem formulation} if 
	\begin{subequations}
		\begin{equation}
		\begin{dcases} V u \delta \in L^1 (\Omega) \textrm{ and } \\
		\int_\Omega u (-\Delta \phi - \vec U \cdot \nabla \phi + V \phi ) =
		\int_{\Omega} f \phi \qquad \forall \varphi \in W_2. \end{dcases}
		\label{eq:vws Brezis}
		\end{equation}%
		When $V$ is only in $L_{loc}^{1}(\Omega )$, we will say that $u $ is a
		``very weak distributional solution'' of \eqref{eq:strong problem formulation}
		if 
		\begin{equation}
		\begin{dcases} V u \delta \in L^1 (\Omega) \textrm{ and } \\
		\int_\Omega u (-\Delta \phi - \vec U \cdot \nabla \phi + V \phi ) =
		\int_{\Omega} f \phi \qquad \forall \varphi \in \mathcal C^2 _c ( \Omega).
		\end{dcases}  \label{eq:vws distributional}
		\end{equation}%
	\end{subequations}
\end{definition}
When $f\in L^{1}(\Omega ,\delta )$ the natural setting for both types of solutions is 
\begin{equation}
u \in L^{n^{\prime },\infty }(\Omega ).
\end{equation}
In our previous paper \cite{DGRT} we proved that:
\begin{theorem}[\cite{DGRT}] \label{t1}
Let $f\in L^{1}(\Omega ,\delta )$ and \eqref{H} hold. Then, there
exists $u \in L^{n^{\prime },\infty }(\Omega )$ such that 
\eqref{eq:vws
distributional} holds. Furthermore if $V\in L^{1}(\Omega ,\delta )$ then %
\eqref{eq:vws Brezis} is satisfied.
\end{theorem}
Moreover, even without ``usual'' boundary conditions (see Remark 9 in \cite{DGRT} for some comments on problem of different nature leading to uniqueness without boundary conditions),  we also proved the following uniqueness result:
\begin{theorem}[\cite{DGRT}]\label{t28}
 There exists, at most, one solution $u $ of \eqref{eq:vws distributional}
such that $\frac{u }{\delta ^{r}}\in L^{1}(\Omega )$, for some $r>1$.
\end{theorem}

One of the main aims of this paper is to show that this exponent $r>1$ is
not optimal in \Cref{t28} because, in fact, $r=1$ suffices. That
improves a remark (following different arguments) pointed out by H. Brezis
to the second author concerning the case $\vec U=\overrightarrow{%
0}$ (see \cite{Gomez-Castro}). Moreover, we shall present here a numerous of
\ other improvements with respect to our previous paper \cite{DGRT}, as,
for instance, the study of the associated eigenvalue problem, the
consideration of flat solutions, the accretiveness in $L^{1}(\Omega ,\delta
^{\alpha })$ of the operator when $\alpha \in \lbrack 0,1)$, \ the
consideration of the associated evolution problem, the confinement for the
solution of the complex Schrödinger problem, etc.

\section{Statement of new existence, uniqueness and regularity results}\label{sec:3}

First, we show the equivalent of the Brezis and distributional formulations,
in the space $L^1 (\Omega, \delta^{-1})$.
\begin{lemma}[equivalence of \eqref{eq:vws Brezis} and \eqref{eq:vws distributional}]
	\label{lem:equivalence} Assume that $f \in L^1 (\Omega, \delta)$, \eqref{H}
	and let $u \in L^{n^{\prime }, \infty} (\Omega) \cap L^1 (\Omega,
	\delta^{-1})$. Then \eqref{eq:vws Brezis} is equivalent to
	\eqref{eq:vws
	distributional}.
\end{lemma}

First we prove an existence result in $L^{n^{\prime },\infty}$ with
additional estimates
\begin{theorem}[General existence result]
	\label{thm:existence result} Assume that $f \in L^1 (\Omega, \delta)$ and %
	\eqref{H}. Then there exists $u \in L^{n^{\prime }, \infty}(\Omega)$
	such that \eqref{eq:vws Brezis} holds. Furthermore, if $f \ge 0$, then $u \ge
	0$. Besides 
	\begin{align}
	\int _ \Omega V |u|\delta \le C_u \int_{\Omega} |f|\delta.
	\end{align}
	where $C_u$ does not depend on $V$ and $f$.
\end{theorem}

Then we will extend our uniqueness result

\begin{theorem}[Uniqueness in $L^1 (\Omega, \protect\delta^{-1})$]
\label{thm:uniqueness L1 Omega delta} Assume that $f \in L^1 (\Omega,
\delta) $ and \eqref{H}. Then, there exists at most one $u \in
L^{n^{\prime }, \infty} (\Omega) \cap L^1 (\Omega, \delta^{-1})$ such that 
\eqref{eq:vws Brezis} holds.
\end{theorem}

From this, several existence and uniqueness results follow.
If the potential is ``very singular'', the condition $V u \delta \in L^1$ acts as boundary condition.
\begin{theorem}\label{t31}
Assume that $f \in L^1 (\Omega, \delta)$, \eqref{H} and $V \ge C \delta^{-2}$
for some $C > 0$. Then there exists a unique $u \in L^{n^{\prime
},\infty} (\Omega) \cap L^1 (\Omega, \delta^{-1})$ such that 
\eqref{eq:vws
Brezis} holds.
\end{theorem}

Better integrability of the data improves the differentiability of the solution and, in particular, the (unique) solution satisfies the Dirichlet condition in the sense that $u \in W_0^1 L^{n',\infty} (\Omega)$.
\begin{theorem} \label{c1t29}
	Assume that $f \in L^1 (\Omega)$ and \eqref{H}. Then, there exists exactly
	one $u \in L^{n^{\prime }, \infty} (\Omega) \cap L^1 (\Omega ;
	\delta^{-1})$ such that \eqref{eq:vws Brezis}. Furthermore, $u \in
	W_0^1 L^{n^{\prime },\infty} (\Omega)$ and 
	\begin{align}
	\int _ \Omega V |u| &\le C \int_ \Omega |f|,\\
	\int_ \Omega V|u| \delta &\le c_\Omega (1 + \| \vec U \|_{L^{n,1}}) \int_ \Omega |f|\delta,\\
	\| \nabla u \|_{L^{n',\infty}} &\le C \int_ \Omega |f|.
	\end{align}
\end{theorem}
The intermediate cases of integrability of the datum $f$ given by the inclusions, for $\alpha \in (0,1)$,
\begin{equation}
	L^1 ( \Omega ) \subset L^1 (\Omega ; \delta ^ \alpha ) \subset L^1 (\Omega ; \delta( 1 + |\log \delta|  )) \subset L^1 (\Omega ; \delta)
\end{equation}
can also be considered. In fact, in \cite{RakoJFA} it was shown that the condition $\frac u \delta  \in L^1 (\Omega)$ is equivalent to the data been in $L^1 (\Omega ; \delta( 1 + |\log \delta|  ))$.

\begin{theorem}
	\label{thm:well posedness in L1 delta 1 + log delta} Assume that $f \in L^1
	(\Omega; \delta ( 1 + |\log \delta| ))$ and  \eqref{eq:H with p greater n}.
	Then there exists a unique $u \in L^{n^{\prime },\infty} (\Omega) $ such that \eqref{eq:vws Brezis}. Furthermore, it is in $	L^1 (\Omega; \delta^{-1})$.
\end{theorem}

When we improve the integrability of $f$ near $\partial \Omega$ we can relax the conditions on $\vec U$.

\begin{theorem}\label{t34}
	Let $0 < \alpha < 1$. Assume that \eqref{eq:H with p greater n}, $f \in
	L^1(\Omega, \delta^\alpha)$ and $\vec U \in L^{\frac n {1-\alpha}} (\Omega)$%
	. Then, there exists a unique solution $u \in L^{n^{\prime },\infty} (\Omega)$ of \eqref{eq:vws Brezis}. Moreover, it is in $
	L^1 (\Omega; \delta^{-1})$. Furthermore,  $u \in W_0^1 L^{ \frac n {n+1 +
	\alpha}}$ and 
	\begin{equation}
		\int_ \Omega V|u|\delta^\alpha \le \int_{\Omega} |f|\delta^ \alpha.
	\end{equation}
\end{theorem}

\section{Proof of the \nameref{lem:equivalence}} \label{sec:4}

The proof is based on the following lemma, which improves \cite{DGRT}. The idea is how well we can approximate a test function $\phi \in W_2$ by functions $\phi_j \in \mathcal C_c ^\infty$. In \cite{DGRT} our approximation was that, for $r > 1$, we can have the convergence of derivatives: $\delta^r D^\alpha \phi_j \to \delta^r D^\alpha \phi$ in $L^\infty$ for $|\alpha|\le 2$ (although this idea is older, see, e.g., Theorem 9.17 in \cite{Brezis libro}). Our improvement here is that, for $r = 1$, we can obtain the same approximation in $L^\infty$-weak-$\star$.  

\begin{lemma}[Approximation of test functions in $W_2$]\label{t3}
Let $\phi \in W_2$. Then, there exists a sequence $\phi_j \in C_c^\infty
(\Omega)$ such that

\begin{enumerate}

\item There exists $C > 0$ such that $\|\nabla \phi_j
\|_{L^\infty} \le C$ for all $j \ge 1$. 

\item $\| \phi_j - \phi \|_{L^\infty} + \| \nabla \phi_j - \nabla
\phi \|_{L^1} \to 0$. 

\item $\delta \Delta \phi_j \rightharpoonup \delta \Delta \phi$ in $%
L^\infty$-weak-$\star$. 

\item $\frac{\phi_j}{\delta }\rightharpoonup\frac{\phi}{\delta }$ in 
$L^\infty$-weak-$\star$. 
\end{enumerate}
\end{lemma}

\begin{proof}
	Following \cite{DGRT}, we shall consider $h\in C^\infty(\R)$ such that 
	$$
		h(t)= 
		\begin{cases} 
			1&\hbox{if\ }t\GEQ2,\\
			0&\hbox{if\ }t\LEQ1,
		\end{cases}
	$$ 
	for $j\in\N^*$ set $\eps=\frac 1j$ and let $h_j(x)=h\left(\frac{\d(x)-\eps}\eps\right),\ x\in\O$. Setting 
	$$
		E_j=\Big\{x\in\O:\dfrac2j\LEQ\d(x)\LEQ\dfrac3j\Big\},
		\qquad  
		E^c_j=\O\backslash E_j.
	$$
	One has the following properties of $h_j$:
	\begin{enumerate}
		\item $\Delta h_j(x)=|\nabla h_j(x)|=0 $ 
		for $x\in E^c_j$,
		\item $\DST h_j(x)\to 1$ as $j \to + \infty$, for any $x\in\O$ (since $h_j(x)=1$ if $\d(x)\GEQ\frac3j$),
		\item $\|\d h_j-\d\|_\infty=\max_{x\in\OV\O}|\d(x)h_j(x)-\d(x)|\LEQ 3(1+\|h'\|_\infty)\eps,$
		\item $\d(x)|\nabla h_j(x)|\LEQ 3\|h'\|_\infty$ and  $\d^2(x) |\Delta h_j(x)|\LEQ c_{h}$  on $\O$, where $c_h$  is constant (depending only on $h$ and $\O$).
	\end{enumerate}
	Let $\phi\in W_2$, the sequence $\f_j =h_j\phi$ is in $C^2_c(\O)$ and enjoy the following property,
	\begin{equation}\label{eq4}
	\hbox{ there is a constant $c>0$ such } \|\nabla\f_j\|_\infty\LEQ c\|\nabla\phi\|_\infty.
	\end{equation}
	Indeed
	\def\meas{{\rm meas\,}}
	$$|\nabla \f_j(x)|\LEQ3\|h'\|_\infty\|\nabla\phi\|_\infty+\|h\|_\infty\|\nabla\phi\|_\infty.$$
	Moreover, one has
	\begin{align}
		\label{eq600}
		\|h_j\phi-\phi\|_{\infty} &\LEQ c\eps \|\nabla\phi \|_\infty,\\
		\label{eq601}
			\int_\O|\nabla\f_j(x)-\nabla\phi(x)|dx &\LEQ c\,\meas \Big\{x\in\O: \d(x)\LEQ\dfrac3j\Big\} \xrightarrow[j\to+\infty]{}0,\\
		\label{eq5}
		|\d(x)\Delta\f_j(x)-\d(x)\Delta\phi(x)| 
		&\LEQ\|\d h_j-\d\|_\infty|\Delta\phi(x)|\hbox{ for }x\in E_j^c.
	\end{align}
	For $x\in E_j$, we have
	\begin{align}
	|\d(x)\Delta\f_j(x)-\d\Delta\phi(x)| &\LEQ \|\d h_j-\d \|_\infty |\Delta\phi(x)|
	+\d^2(x)\|\nabla\phi\|_\infty |\Delta h_j(x)| \nonumber \\
	&\quad +2\d(x)|\nabla h_j(x)\||\nabla\phi\|_\infty.\label{eq6}
	\end{align}
	The statements  (\ref{eq5}) and (\ref{eq6})  are obtained with a straightforward computation. From those statements, we deduce that  there is a constant $c_\phi>0$ such that
	\begin{equation}\label{eq7}
	\|\d\Delta\f_j-\d\Delta\phi\|_\infty\LEQ c_\phi.
	\end{equation} 
	Since $$\meas(E_j)\xrightarrow[j\to+\infty]{}0\hbox{ and }\|\d h_j-\d\|_\infty\xrightarrow[j\to+\infty]{}0$$
	we have
	\begin{eqnarray}\label{eq9}
	\int_\O|\d(x)\Delta\f_j (x)-\d(x)\Delta\phi (x)| dx
	&\LEQ&\int_{E^c_j}|\d(x)\Delta\f_j(x)-\d(x)\Delta\phi(x)|dx+c_\phi\meas(E_j)\nonumber\\
	&\LEQ&\|\d h_j-\d\|_\infty\|\Delta\phi\|_\infty+c_\phi\meas(E_j)\xrightarrow[j\to+\infty]{}0. \label{eq8}
	\end{eqnarray}
	One deduces from relations (\ref{eq7}) and (\ref{eq8}) that 
	$$ 
		\d\Delta\f_j \rightharpoonup \d\Delta\phi \quad 
		\textrm{ weakly-$\star$ in $L^\infty(\O)$.}
	$$
	Since  $C_c^\infty(\O)$ is dense in $C^2_c(\O)$, we obtain the desired result.
\end{proof}

	With this technique we can now move the proof of the equivalence.

\begin{proof}[Proof of \Cref{lem:equivalence}]
	Let $\phi$ be in $W_2$. Then, we have a sequence $\phi_j\in C^\infty_c(\O)$ with the convergence stated in \Cref{t3} such that
	\begin{equation}\label{eq300}
	\int_\O\g\Big[-\Delta\phi_j+\vec U\cdot\nabla\phi_j+V\,\phi_j\Big]dx=\int_\O f\,\phi_j dx.
	\end{equation}
	Therefore, we have
	\begin{equation}\label{eq301}
	\lim_{j\to+\infty}\int_\O\g \Delta\phi_jdx=\lim_j\int_\O\dfrac\g\d(\d\Delta\phi_j)dx=\int_\O\g\Delta\phi \,dx,
	\end{equation}
	since $\dfrac\g\d\in L^1(\O)$ and $\DST\d\Delta\phi_j {\rightharpoonup} \d\Delta\phi$ in $L^\infty(\O)$-weak-$\star$ as $j \to \infty$.\\
		
	For the same reason, one has:
	$$
		\lim_j\int_\O\g\,\vec U \cdot\nabla \phi_j dx=\int_\O\g\,\vec U \cdot\nabla \phi dx
	$$
	since $u \vec U\in L^1$ and $\nabla\phi_j\rightharpoonup\nabla\phi$ in $L^\infty$-weak-$\star$. Moreover,
	$$\lim\int_\O\g\,V\phi_jdx=\int_\O\g\,V\phi dx$$
	(since $V\,\g\d\in L^1(\O)$ and $\frac{\phi_j}\d\rightharpoonup\frac\phi\d$  in $L^\infty(\O)$-weak-$\star$). 
	We easily pass to the limit in equation (\ref{eq300}) and thus $u$ satisfies \eqref{eq:vws Brezis}.
\end{proof}

\section{Proof of the existence and regularity results}\label{sec:5}
We will consider the approximating sequence 
\begin{equation}\label{eq31-4}
\begin{cases}
-\Delta\g_j+\vec U_j\cdot\nabla \g_j+V_j\g_j=f_j\\
\g_j\in W_0^{1,1}(\O)\cap W^2L^{p,1}(\O)\end{cases}
\end{equation}
i.e. 
\begin{equation}  \label{eq:weak formulation of approximating problem}
\int_ \Omega u_j (-\Delta \varphi - \vec U_j \cdot \nabla \varphi + V_j
\varphi) = \int_ \Omega f_j \varphi \qquad \forall \varphi \in W_2 .
\end{equation}
where 
\begin{align}
V_j (x) &= \min ( V(x), j ) , \\
f_j (x) &= \sign (f(x)) \min (|f(x)|, j)
\end{align}
and $\vec U_j \in \mathcal{C}_c^\infty (\Omega)^n$, such that \eqref{eq:incompressible strong} and 
\begin{equation}
\vec U_j \to \vec U \text{ in } L^{p,1} (\Omega)^n.
\end{equation}

First we recall our result in \cite{DGRT} about the approximation of solutions

\begin{theorem}[existence and approximation of solutions when $f \in L^1 (\Omega; \protect%
\delta)$] \label{t11}
	\label{prop:approximation of solution L1 omega delta} Assume $f \in L^1
	(\Omega, \delta)$ and \eqref{H}. Then, there is a unique solution $u_j
	\in W_0^{1,1} (\Omega) \cap W^2 L^{p,1} (\Omega) $ of 
	\eqref{eq:weak
	formulation of approximating problem} and there exists $u$ such that:
	
	\begin{enumerate}
	\item {$u$ is a solution of \eqref{eq:vws distributional} },
	
	\item {$u_j \to u$ a.e. in $\Omega$, }
	
	\item {$u_j \rightharpoonup u$ in $L^{n^{\prime }, \infty}$%
	-weak-$\star$ and $W^{1,q} (\Omega, \delta)$-weak, for $q < n'$},
	
	\item {$u_j \to u$ in $L^r (\Omega)$ for $r < n^{\prime }$, 
	}
	
	\item $u_j \vec U_j \to u \vec U$ in $L^1(\Omega)^n$,
	\item $\int_\O V_j|\g_j|\d dx\LEQ c (1+\|\vec U_j\|_{L^{n,1}})\int_\O|f_j|\d dx$,
	\item $V_j\g_j\d\rightharpoonup V\,\g\d$ weakly in $L^1_{loc}(\O)$.
	\end{enumerate}
\end{theorem}
We can make some additional estimates if we restrict the set of datum $f$ to $L^1 (\Omega)$:
\begin{proposition}[existence of solutions when $f \in L^1 (\Omega)$]
	\label{t12}
	Assume that $f \in L^1 (\Omega)$ and \eqref{H}. Then, the sequence $u_j$
	satisfies 
	\begin{align}
	\|\nabla u_j \|_{L^{n^{\prime },\infty}} &\le C \int_{ \Omega } |f_j|,
	\label{eq:bound Sobolev L n prime with L1 of f} \\
	\int_ \Omega V_j |u_j| &\le C \int_{ \Omega } |f_j|.
	\label{eq:bound V omega in L n prime with L1 of f}
	\end{align}
	Hence 
	\begin{equation}
	u_j \rightharpoonup u \text{ in } W_0^1 L^{n^{\prime },\infty}
	(\Omega),
	\end{equation}
	and the equations \eqref{eq:bound Sobolev L n prime with L1 of f} and 
	\eqref{eq:bound
	V omega in L n prime with L1 of f} hold for $u, V$ and $f$.
\end{proposition}

\begin{proof}
	Let $k>0$. Then the sequence given in \Cref{t11} satisfies
	\begin{equation}\label{eq12}
	\int_\O\vec U_j\cdot\nabla\g_jT_k(\g_j)dx=0\hbox{ and }\int_\O V_j\g_jT_k(\g_j)dx\GEQ0.
	\end{equation}
	Therefore, we can use $T_k(\g_j)$ as a test function in equation (\ref{eq31-4}) and derive
	\begin{equation}\label{eq13}
	\int_\O|\nabla T_k(\g_j)|^2dx\LEQ k\int_\O|f_j|dx\LEQ k\int_\O|f(x)|dx.
	\end{equation}
	From relation (\ref{eq13}), we deduce (see \cite{BBGGPV} or \cite{RakoBook}) that
	\begin{equation}\label{eq14}
	\|\nabla\g_j\|_{L^{n',\infty}}\LEQ c|f|_{L^1(\O)}.
	\end{equation} 
	While to obtain relation \eqref{eq:bound V omega in L n prime with L1 of f}, we choose as a test function for $t>0$,
	$$\Phi(t;\g_j)= (|\g_j|-t)_+  \sign(\g_j).$$
	Knowing as before that
	\begin{equation}\label{eq15}
	\int_\O\vec U_j\cdot\nabla\g_j\Phi(t;\g_j)dx=0
	\end{equation}
	one obtains from equation (\ref{eq31-4}) that
	\begin{equation}\label{eq16}
	\int_{|\g_j|>t}|\nabla\g_j|^2dx+\int_\O V_j\g_j\Phi(t;\g_j)dx=\int_\O f_j\Phi(t,\g_j)dx.
	\end{equation}
	We derive with respect to $t$ this equation
	\begin{equation}\label{eq17}
	-\dfrac d{dt}\int_{|\g_j|>t}|\nabla \g_j|^2dx+\int_{|\g_j|>t} V_j|\g_j|dx
	=\int_{|\g_j|>t}f(x)\sign(\g_j)dx.
	\end{equation}
	Since the first term is non negative, we conclude from relation (\ref{eq17}) that, for all $t>0$,
	\begin{equation}\label{eq18}
	\int_{|\g_j|>t}V_j|\g_j|dx\LEQ\int_{|\g_j|>t}|f(x)|dx.
	\end{equation}
	Letting $t\to0$, we get the desired relation \eqref{eq:bound V omega in L n prime with L1 of f}.
	Since
	$V_j\g_j\to V\g$ a.e. in $\O$, Fatou's lemma yields 
	\begin{equation}\label{eq19}
	\int_\O V|\g|dx\LEQ\int_\O|f(x)|dx.
	\end{equation}
	Given that $\nabla\g_j\rightharpoonup\nabla\g$ in $L^{n',\infty}$-weak-$\star$, we derive
	\begin{equation}\label{eq20}
	\|\nabla \g\|_{L^{n',\infty}}\LEQ c|f|_{L^1(\O)}.
	\end{equation}
	That \eqref{eq:vws Brezis} is satisfied is a consequence of Lemma \ref{lem:equivalence}, since, by the Hardy's inequality, we have
	\begin{equation}\label{eq260}
	\left|\dfrac\g\d\right|_{L^1(\O)}\LEQ c\|\nabla \g\|_{L^{n',\infty}}<+\infty.
	\end{equation}
	This concludes the proof.
\end{proof}

With this we proceed
\begin{proof}[Proof of \Cref{c1t29}]
	According to \Cref{t12}, the sequence $\g_j$ belongs to a bounded set of $W^1_0L^{n',\infty}(\O)$ and since the sequence converges to a solution $\g$ of the equation \eqref{eq:vws distributional} given in \Cref{t1}, we deduce that this solution $\g$ is in $W^1_0L^{n',\infty}(\O)$ and satisfies the same kind of estimates as $\g_j$. Moreover, $\frac \g\d\in L^1(\O)$ according to relation (\ref{eq260}). Now we may appeal \Cref{thm:uniqueness L1 Omega delta}  to conclude that $\g$ is unique.   
\end{proof}
Finally we can prove
\begin{proof}[Proof of \Cref{thm:existence result}]
	Let $f$ be in $L^1(\O;\d)$ and consider $f_j=\sign\big(f(\cdot)\big)\min\big(|f|;j\big),\ j\GEQ0$. Then according to the above result \Cref{c1t29}, there exists 
	a unique $\WT\g_j\in W_0^1L^{n',\infty}(\O)$ satisfying
	\begin{equation} \tag*{$(\ref{eq:vws Brezis})_j$} \label{eq2j}
	\int_\O\WT\g_j\big[-\Delta\phi-\vec U\cdot\nabla\phi+V\,\phi\big]dx=\int_\O f_j\phi dx,\quad\forall\,\phi\in W_2. 
	\end{equation}
	Since $f_j-f_k\in L^1(\O)$ for $k$ and $j$ in $\N$, by the same corollary 1 of Theorem \ref{thm:uniqueness L1 Omega delta} and Theorem~\ref{t12},
	we deduce that $\WT\g_j-\WT\g_k$ is the unique solution of 	\begin{equation*}
	\int_\O(\WT\g_j-\WT\g_k)\big[-\Delta\phi-\vec U\cdot\nabla\phi+V\,\phi\big]dx=\int_\O (f_j-f_k)\phi dx,\quad\forall\,\phi\in W_2,
	\end{equation*}then it satisfies
	\begin{align}
	\DST\int_\O V|\WT\g_j-\WT\g_k|\d\,dx&\LEQ c_u\int_\O|f_j-f_
	k|\d\,dx\nonumber\\
	\intertext{and}
	\|\WT\g_j-\WT\g_k\|_{L^{n',\infty}}&\LEQ c_u\int_\O|f_j-f_k|\d\,dx.\label{eq29 }
	\end{align}
	Thus $(\WT\g_j)_j$ is a Cauchy sequence in $L^{n',\infty}(\O)$ and $(V\,\WT\g_j)_j$ is also a Cauchy one in $L^1(\O;\d)$. Therefore one has easily $\WT \g\in L^{n',\infty}(\O)$ with $V\,\WT\g\in L^1(\O;\d)$ such that $\WT\g$ satisfies equation \eqref{eq:vws Brezis}. Moreover,
	$\int_\O V|\WT\g|\d\,dx\LEQ c\int_\O f\d\,dx$ and if $f\GEQ0$ then $f_j\GEQ0$ therefore $\WT\g_j\GEQ0$ which yields that $\WT \g\GEQ0$.
\end{proof}

\section{Proof of the uniqueness results}\label{sec:6}

To complete the proof of the results above we only need to prove the uniqueness of the solutions of the
equations. Once we complete the proof of \Cref{thm:uniqueness L1 Omega delta}
the rest of the proofs will follow as a corollary. The main tool in this
proof will be a Kato type inequality up to the boundary.

\subsection{Kato's inequality}

Notice that, in the following result no Sobolev space is included, and hence no trace is involved. We do not consider boundary conditions in the usual way.
\begin{theorem}[Variant of Kato's inequality]\label{t4}
	Let $\OV \g$ be in $W^{1,1}_{loc}(\O)\cap L^{n',\infty}(\O)$ with $\frac{\OV\g}\d\in L^1(\O)$ and $\vec U\in L^{n,1}(\O)^n$ with $\div(\vec U)=0$ in $\calD'(\O),\ \vec U\cdot\vec \nu=0$ on $\p\O$.
	Assume that $L\OV\g=-\Delta\,\OV\g+\div(\vec U\,\OV\g)\in L^1(\O;\d).$
	Then, for all $\phi\in W_2$, $\phi\GEQ0$ one has
	\begin{enumerate}
		\item $\DST\int_\O\OV\g_{+}L^*\,\phi\, dx\LEQ\int_\O\phi\,\sign_{+}(\OV\g)\,L\,\OV\g\, dx$,
		\item $\DST\int_\O|\OV\g|L^*\,\phi\, dx\LEQ \int_\O\phi\,\sign(\OV\g)L\OV\g\,dx,$ 
	\end{enumerate}
	where $L^*\phi=-\Delta\phi-\vec U\cdot\nabla\phi=-\Delta\phi-\div(\vec U\,\phi),$ 
	$$
		\sign_+(\s)=\begin{cases}1&\hbox{if }\s>0,\\0&\hbox{otherwise},\end{cases}
		\quad 
		\textrm{ and }
		\quad 
		\sign(\s)=\begin{cases}1&\hbox{if }\s>0,\\0&\hbox{if }\s=0,\\-1&\hbox{if }\s<0.\end{cases}
	$$ 
\end{theorem}
The proofs of both theorem (\Cref{thm:uniqueness L1 Omega delta} above and \Cref{t4}  below) follow the same argument as we did in \cite{DGRT} (Corollary 4 Theorem 10, Theorem 8). The only difference is the use of the new approximation \Cref{t3}. For the convenience of the reader we sketch here those proofs :

\begin{proof} [Sketch of the proof of \Cref{t4}]
	Let $\phi\GEQ0$, $\phi\in W_2$. Then according to \Cref{t3} one has a sequence $\phi_j\in C^\infty_c(\O)$ such that
	$\d\Delta\phi_j\rightharpoonup \d\Delta\phi$ in $L^\infty(\O)$-weak-$\star$. This implies, together with the hypothesis that $\frac{\g_+}\d\in L^1(\O)$, that
	\begin{equation}\label{eq23}
	\lim_{j\to+\infty}\int_\O\OV\g_+\Delta\phi_j dx=\int_\O\OV\g_+\Delta\phi dx.
	\end{equation}
	For the same reason
	\begin{equation}\label{eq24}
	\lim_{j\to+\infty}\int_\O\vec U\cdot\nabla\phi_j\OV\g_+dx=\int_\O\vec U\cdot\nabla\phi\OV\g_+dx.
	\end{equation}
	We conclude as in \cite{DGRT}, knowing that the local Kato's inequality (Theorem 10 in \cite{DGRT}) holds true.
\end{proof}

One of the consequence of the Kato's inequality is the following maximum principle.

\begin{corollary}[of Theorem \ref{t4}]\label{c1t4}
	Under the same hypothesis as for \Cref{t4}, assume that $L\OV\g=f(x)-G(x;\OV\g)\in L^1(\O;\d)$, with $G:\O\times\R\to\R$ a Caratheodory function (i.e for a.e $x$, $\s\to G(x;\s)$ is continuous, and $x\to G(x;\s)$ is measurable  $\forall\,x$), satisfying the sign-function condition   $$\sign(\s)G(x;\s)\GEQ0\quad \forall\,\s\in\R\hbox{ a.e }x\in\O.$$ Then,
	if $f\LEQ0$ one has $\OV\g\LEQ0$.
\end{corollary}
\begin{proof} 
	Let $ \phi\in W_2$ be such that  $\phi\GEQ0$. Then
	\begin{equation}\label{eq101}
	\int_\O  \OV\g_+L^*\phi\,dx\LEQ\int_\O\phi\,\sign_+(\OV\g)f(x)dx-\int_\O\phi G(x;\OV\g_+)dx,
	\end{equation}
	since $G(x;0)=0$ and $\sign_+(\s)G(x;\s)=G(x;\s_+)\GEQ0$.
	Therefore, from this last inequality (\ref{eq101}), knowing that
	$$-\phi G(x;\OV\g_+)\LEQ0,\quad f(x)\sign_+(\OV\g)\LEQ0,$$
	we deduce that
	\begin{equation}\label{eq266}\forall\,\phi\GEQ0,\ \phi\in W_2:\int_\O \OV\g_+L^*\phi dx\LEQ 0.
	\end{equation}
	Since $\OV \g\in L^{n',\infty}(\O)$ and $L^*\phi=-\Delta\phi-\vec U\cdot \nabla\phi$ is in $L^{n,1}(\O)$ for $\phi\in W^2L^{n,1}(\O)\cap H^1_0(\O)$, thus a density argument leads from equation (\ref{eq266}) to 
	\begin{equation}\label{eq27}
	\int_\O\OV\g_+L^*\phi dx\LEQ 0\qquad  \forall\,\phi\in W^2L^{n,1}(\O)\cap H^1_0(\O),\ \phi\GEQ0.
	\end{equation}
	Thus, we get:
	$$\OV\g_+=0.$$
	This completes the proof.
\end{proof}
\subsection{Proof of the uniqueness results}

\begin{proof} [Proof of \Cref{thm:uniqueness L1 Omega delta}]
	Let $\OV \g=\g_1-\g_2$ where $\g_i$ are in $L^{n',\infty}(\O)\cap L^1(\O;\d^{-1})$ and are two solutions of equation  \eqref{eq:vws Brezis} (or \eqref{eq:vws distributional}, these formulations are equivalent due to  \Cref{lem:equivalence} since $\g_i\in L^1(\O;\d^{-1})$). Then
	$$L\OV \g=-V\OV\g\in L^1(\O;\d).$$
	From \Cref{t4} one has, for a test function $\phi\in W_2$ such that $\phi\GEQ0$,
	\begin{equation}\label{eq26}
	\int_\O|\OV\g|L^*\phi dx\LEQ-\int_\O\phi\sign(\g)V\OV\g=-\int_\O\phi V|\OV\g|dx\LEQ0.
	\end{equation}
	As before one has:
	\begin{equation}\label{eq27a}
	\int_\O|\OV\g|L^*\phi dx\LEQ 0\qquad \forall\,\phi\in W^2L^{n,1}(\O)\cap H^1_0(\O),\ \ \phi\GEQ 0.
	\end{equation}
	Considering $\OV\phi_0\in W^2L^{n,1}(\O)\cap H^1_0(\O),\ \OV\phi_0\GEQ0$ solution of $L^\star\OV\phi_0=1$, we deduce $$\DST\int_\O|\OV \g|dx\LEQ0$$ 
	thus $\OV\g=0$.
\end{proof}

\begin{proof}[Proof of \Cref{thm:well posedness in L1 delta 1 + log delta}]
		First let us assume that $f \ge 0$.	
		Since $f$ is a nonnegative function in $L^1(\O;\d)$, the existence of a solution $\g\GEQ0$ is a consequence of Theorem \ref{thm:existence result}. To prove the uniqueness result, let us show that exists a $c>0$ independent of $\g, f$ and $V$ such that
		\begin{equation}\label{eq30}
		\int_\O\dfrac\g\d dx+\int_\O V\,\g\d(1+|\Log\d|)dx\LEQ c\int_\O f(x)(1+|\Log\d|)\d dx.
		\end{equation}
		For this, we use the argument introduced in \cite{RakoJFA} by choosing as a test function $$\phi=\f_1\Log(\f_1+\eps),\ \eps>0,$$ 
		where $\f_1$ the first eigenfunction of $-\Delta$ with homogeneous Dirichlet boundary condition.\\
		One obtains
		\begin{equation}\label{eq31}
		-\!\!\int_\O\!\g\Delta(\f_1\Log(\f_1+\eps))dx-\!\!\int_\O\!\vec U\,\g\cdot\nabla(\f_1\Log(\f_1+\eps))dx
		+\!\int_\O\! V\g\f_1\Log(\f_1+\eps)dx=\!\!\!\int_\O f\f_1\Log(\f_1+\eps)dx.
		\end{equation}
		We develop  each term in relation (\ref{eq31}) as we did in \cite{RakoJFA} knowing that $\f_1$ is equivalent to the distance function (say $\exists c_0>0,\ c_1>0,\ c_0\d\LEQ\f_1\LEQ c_1\d$). We derive
		\begin{align}\label{eq32}
		\int_\O|\nabla\f_1|^2\dfrac\g{\f_1+\eps}dx&-\int_\O V\,\g\f_1\Log(\f_1+\eps)dx\\
		&\LEQ c\left[\int_\O\g(x)dx+\int_\O f(x)(1+|\Log\d|)\d dx\right]\nonumber\\
		&\quad +c\int_\O\|\vec U\|\,|\Log\d|\g dx+c\int_\O\|\vec U\|(x)\g(x)dx.\nonumber
		\end{align}
		Since $\vec U\in L^{p,1}(\O),\ p>1$ then $\|\vec U\|\Log \d\in L^{n,1}(\O)$ and there exists a constant $c>0$. 
		$$\Big\|\,|\vec U|\Log\d\Big\|_{L^{n,1}}\LEQ c\|\vec U\|_{L^{p,1}(\O)}.$$
		Therefore, we have
		\begin{equation}\label{eq33}
		c\int_\O\|\vec U\|\,|\Log\d|\g\,dx+c\int_\O\|\vec U\|(x)\g(x)dx
		\LEQ c_U\|\g\|_{L^{n',\infty}}\LEQ c\int_\O f(x)\d(x)dx.
		\end{equation}
		From relations (\ref{eq32}) and (\ref{eq33}), we deduce 
		\begin{equation}\label{eq34}
		\int_\O|\nabla\f_1|^2\dfrac\g{\f_1+\eps}dx-\int_\O V\,\g\f_1\Log(\f_1+\eps)dx\LEQ
		c\int_\O f(x)(1+|\Log\d|)\d dx.
		\end{equation}
		As in \cite{RakoJFA} we write
		\begin{equation}\label{eq35}
		\int_\O V\,\g\f_1|\Log(\f_1+\eps)|dx=-\int_\O V\,\g\f_1\Log(\f_1+\eps)dx+2\int_{\f_1+\eps>1}\!\!V\,\g\f_1\Log(\f_1+\eps)dx.
		\end{equation}
		Combining these two last relations, we get
		\begin{equation}\label{eq36}
		\int_\O|\nabla\f_1|^2\dfrac\g{\f_1+\eps}dx+\int_\O V\,\g\f_1|\Log(\f_1+\eps)|dx\
		\LEQ c\int_\O f(x)(1+|\Log\d|)\d dx+c\int_\O V\,\g\d dx.
		\end{equation}
		Noticing that in a neighborhood of the boundary $\p\O\subset U\subset\OV\O$ one has $\DST\inf_{x\in U}|\nabla\f_1|^2(x)>0$, we derive from relation (\ref{eq36}) the inequality (\ref{eq30}).\\

		Let $f$ be in $L^1\big(\O;\d(1+|\Log\d|)\big)$, we decompose  $f=f_+-f_-$ where $f_+, f_- \ge 0$. Due to the first part of the proof, we have $\g_1 $ (resp. $\g_2$) a nonnegative solution of \eqref{eq:vws Brezis} associated to $f_+$ (resp. $f_-$). One has according to relation (\ref{eq30}) for $i=1,2$
		\begin{equation}\label{eq37}
		\int_\O\dfrac{\g_i}\d dx+\int_\O V\,\g_i\d(1+|\Log\d|)dx\LEQ
		c\int_\O|f|(1+|\Log\d|)\d dx.
		\end{equation}
		By linearity we deduce that $\WT\g=\g_1-\g_2$ is a solution of equation \eqref{eq:vws distributional} and satisfies $\frac {\WT\g}\d\in L^1(\O)$. We conclude with Theorem \ref{thm:uniqueness L1 Omega delta} to obtain the result.
	\end{proof}

\section{Estimates when the datum $f$ is $L^1(\O;\d^\a),\ 0\LEQ\a\LEQ1$}\label{sec:7}
\begin{lemma}\label{l3}
	Under the same assumptions as for \Cref{thm:well posedness in L1 delta 1 + log delta}, if furthermore $f\in L^1(\O;\d^\a)$, $0\LEQ\a<1$ then the function $\WT\g$ solution of equation \eqref{eq:vws Brezis} verifies
	$$\int_\O(V|\WT\g|\d^\a)(x)dx\LEQ c_\a\int_\O|f(x)|\d^\a(x)dx.$$
\end{lemma}
\begin{proof} 
	For $k\GEQ 0$, let us consider $V_k=\min(V;k)$ and define  the linear operator $T_k$ on $L^1(\O;\d)$ by setting $T_kf=V_k\WT\g_{kf}$, where $\WT\g_{kf}$  is the unique solution of
	\begin{equation}\label{eq38}
	\int_\O \WT\g_{kf}\left[-\Delta\phi+\vec U\cdot\nabla\phi+V_k\phi\right]dx=\int_\O f\phi dx\qquad \forall\,\phi\in W_2.
	\end{equation}
	The existence and uniqueness follows from Theorem 7 in \cite{DGRT}.
	
	According to Corollary \ref{c1t29} of Theorem \ref{thm:uniqueness L1 Omega delta} and \Cref{t12}. 
	$T_k$ maps $L^1(\O)$ into itself with
	\begin{equation}\label{eq39}
	\|T_k f\|_{L^1(\O)}=\int_\O V_k|\WT\g_{kf}|dx\LEQ\|f\|_{L^1(\O)},
	\end{equation}
	and $T_k$ maps $L^1(\O;\d)$ into itself with
	\begin{equation}\label{eq40}
	\|T_kf\|_{L^1(\O;\d)}\LEQ c(1+\|\vec U\|_{L^{n,1}})\|f\|_{L^1(\O;\d)}.
	\end{equation}
	Since $L^1(\O;\d^\a)$ is the interpolation space in the sense of Peetre between $L^1(\O;\d)$ and $L^1(\O)$, that is 
	$$L^1(\O,\d^\a)=\Big(L^1(\O;\d),L^1(\O)\Big)_{\a,1},$$
	we derive from Marcinkewicz's interpolation theorem (see \cite{Bennet-Sharpley,RakoBook}) that $T_k$ maps $L^1(\O;\d^\a)$ into itself and
	$$\|T_kf\|_{L^1(\O;\d^\a)}\LEQ c^\a(1+\|\vec U\|_{L^n,1})^\a\|f\|_{L^1(\O,\d^\a)},\quad\forall f\in L^1(\O;\d^\a).$$
	Considering the unique solution $\WT\g_{kj}$ for $j$ fixed in $\N$, of the equation 
	\begin{equation}
	\int_\O\WT\g_{kj}\Big[-\Delta\phi-\vec U\cdot\nabla\phi+V_k\phi\Big]dx=\int_\O f_j\phi dx, \quad \forall \varphi \in W_2,   \tag*{$(\ref{eq:vws Brezis})_{kj}$}
	\end{equation}
	where $f_j=\sign(f)\min(|f|,j)$, applying \Cref{t11} with the sequence $(\WT\g_{kj})_k$, and due to the uniqueness result we deduce that, when $k\to+\infty,\ \WT\g_{kj}\to \WT\g_j$  in $L^{n',\infty}(\O)$ and $\WT\g_j$ is the solution of \ref{eq2j}. Therefore, one has
	\begin{equation}\label{eq41}
	\int_\O V|\WT\g_j|\d^\a dx\LEQ\lim_{k\to+\infty}|T_kf_j|_{L^1(\O;\d^\a)}\LEQ c_\a|f_j|_{L^1(\O;\d^\a)}.
	\end{equation}
	As we have shown in the proof of Theorem \ref{thm:existence result}, $\WT\g_j$ converges to $\WT\g$ as $j\to+\infty$; we deduce the desired inequality.
\end{proof}

The proof of \Cref{t34} needs the following lemma given in Theorem 13 of \cite{DGRT}.
\begin{lemma}\label{l4}
	Let $0<\a<1,\ g\in L^1(\O;\d^\a),\ \vec U$ in $L^{\frac n{1+\a}}(\O)^n$, \eqref{eq:incompressible strong}. Then, there exists a unique solution $\OV\g\in L^{n',\infty}(\O)$ satisfying
	\begin{equation}\label{eq42}
	\int_\O\OV\g\Big[-\Delta\phi-\vec U\cdot\nabla\phi\Big]dx=\int_\O g\phi dx\qquad\forall\,\phi\in W_2.
	\end{equation}
	Moreover, there exists a constant $K(\a;\O)>0$ such that 
	\begin{equation}\label{eq43}
	\|\OV\g\|_{W_0^1L^{\frac n{n-1+\a}}(\O)}\LEQ K(\a;\O)\left(1+\|\vec U\|_{L^{\frac n{1-\a}}}\right)|g|_{L^1(\O;\d^\a)}.
	\end{equation}
\end{lemma}

\begin{proof}[Proof of \Cref{t34}]
	Let $\g$ be the unique solution (2) given by Theorem \ref{thm:well posedness in L1 delta 1 + log delta} when $f\in L^1(\O;\d^\a)$, $0<\a<1$. We set $g=V\, \g-f$. Then following Lemma \ref{l3}, one has $g\in L^1(\O;\d^\a)$ and $\g$ satisfies the same type equation (\ref{eq42}). Therefore, we can apply Lemma \ref{l4} to conclude.
\end{proof}

\section{Some consequences: principal eigenvalue and eigenfunction of $-\Delta +\vec U\cdot \protect\nabla $ and of the operator $A$, the $m$
-accretivity of $A$ and the complex Schrödinger problem in the whole space}
\label{sec:8}
\subsection{Principal eigenvalue and eigenfunction for $-\Delta+\vec U\cdot \nabla $ and the $m$-accretivity of $-\Delta+\vec U\cdot \nabla+V$}
\label{subsec:8.1}
Let us start by recalling a well-known result (see, e.g., \cite{Daners})
\begin{theorem}[Krein-Rutman's theorem]\label{tKR}
	Let $X$ be an ordered Banach space, the interior positive cone $K$ of which $\mathring K$ is non void, $T:X\to X$ a compact linear operator which is strongly positive, i.e $Tf>0$ if $f>0$.
	Then, the spectral radius of $T,\ r(T)>0$ and is a simple eigenvalue  with an eigenvector $\psi_1 \in \mathring K$.
\end{theorem}
We recall the following definition of an $m$-accretive operator.
\begin{definition}[$m$-accretive operator]
	Let $X$ be a Banach space. A linear unbounded operator $$A:D(A)\subset X\to X$$
	is called accretive if
	\begin{enumerate}
		\item  $\forall\,\WT\g\in D(A)$ and $\forall\,\lambda>0$ it holds that $ \|\WT\g\|_X\LEQ\|\WT\g+\lambda A\,\WT \g\|_X$.
	\end{enumerate}
	The operator is called $m$-accretive if it is accretive and
	\begin{enumerate} [resume]
				\item $\forall \lambda>0$ we have that $ \overline{D(A)} \subset R(I+\lambda A)$.
	\end{enumerate}
\end{definition}

Let us consider $\vec U\in L^{p,1}(\O)^n,\ p>n$ (or in $L^{n,1}(\O)^n$ but with a small norm as in \cite{DGRT}), we define a compact operator $$T:C(\OV\O)\to W^1_0L^{p,1}(\O)\hookrightarrow C(\OV \O)$$ by setting
$$
	Tf=\g
	\hbox{ if and only if }
	\begin{cases}
		-\Delta\g-\vec U\cdot \nabla \g=f
		\\
		\g\in W^1_0L^{p,1}(\O),p>n
	\end{cases}
$$
(the existence, uniqueness and regularity of $\g$ in given in \cite{DGRT}).
Using the Bony's maximum principle or Stapamcchia's argument, we have for $f>0$ the solution $\g>0$.
Since the positive cone $K=C_+(\OV\O)=\{\f\in C(\OV\O):\f\GEQ0 \}$ has its interior $\mathring K$ non void, we may apply the Krein-Rutman's theorem (see Theorem \ref{t32}) to derive the
\begin{theorem}\label{t36}
	There exist a real $\lambda_1>0$ and a positive function $\psi_1\in W^2L^{p,1}(\O)\cap H^0_1(\O)$ such that 
	$$-\Delta\psi_1-\vec U\cdot \nabla \psi_1=\lambda_1\psi_1.$$
	Moreover, $L^1(\O;\d)\hookrightarrow L^1(\O;\psi_1)$ and if $\vec U\in L^\infty(\O)^n$ then $\psi_1\GEQ c\d$ so that $$L^1(\O;\d)=L^1(\O;\psi_1).$$
\end{theorem}
\begin{remark}
	The fact that $L^1(\O;\d)\hookrightarrow L^1(\O;\psi_1)$ comes from the fact 
	$$0<\psi_1(x)\LEQ\d(x)\|\nabla\psi_1\|_\infty\LEQ c\|\psi_1\|_{W^2L^{p,1}}\d(x)<+\infty,\ x\in\O.$$
\end{remark}
Next, we want to prove Theorem \ref{t32} concerning the $m$-accretivity of $A=-\Delta+\vec U\cdot \nabla+V$ in the Banach space $L^1(\O;\d^\a),\ 0\LEQ\a\LEQ1$.
The argument is similar to the one given in \cite{RakoADE}.\\
First, we endow the space $L^1(\O;\d^\a)$ with the following equivalent norm 
$$\|f\|_\a=\int_\O|f(x)|\psi^\a_1(x)dx,$$
with $\psi_1$ given in Theorem \ref{t36}.
We shall introduce the following definition
\begin{definition}\label{d36}
	Let $\OV\g$ be in $L^1(\O,;\d^\a)$ with $ V\OV\g\in L^1(\O;\d^\a)$. We will say that $A\OV\g\in L^1(\O;\d^\a)$ if there exists a function $f\in L^1(\O;\d^\a)$ such that $A\OV\g=f$ and 
	\begin{equation}\label{eq45}
	\int_\O\phi fdx=\int_\O\OV\g\Big[-\Delta\phi-\vec U\cdot \nabla \phi+V\phi\Big]dx,\quad\forall\,\phi\in C^2_c(\O).
	\end{equation}
\end{definition}
Here, $V\GEQ0$ locally integrable and $\vec U$ is as in Theorem \ref{t1}.
When $\vec U = 0$ and $0 \le V \in L^\infty (\Omega)$ then we choose $D(A) \subset W_0^{1,1} (\Omega)$. In this setting, in which traces exist, previous results apply (see, e.g., \cite{Fila-Souplet-Weissler}). However, when $V \ge c \delta^{ -2}$ (our main case of interest due to the Schrödinger equation) we can no longer expect that $D(A) \subset W_0^{1,1} (\Omega)$. Nonetheless, we have shown that $D(A) \subset L^1 (\Omega; \delta^{-1})$, a space which \emph{also} acts as having a Dirichlet boundary condition on $\partial \Omega$.\\

We can define the operator $A:D(A)\subset L^1(\O;\d^\a)\to L^1(\O;\d^\a)$, where the domain of $A$ is
$$D(A)=\Big\{\OV \g\in L^{n',\infty}(\O)\cap L^1(\O;\d^{-1})\cap L^1(\O;V\d):A\OV\g\in L^1(\O;\d^\a)\Big\}.$$
Therefore, we always have $C^2_c(\O)\subset D(A)\subset L^1(\O;\d^\a)$ this implies that $D(A)$ is dense in $L^1(\O;\d^\a),\ 0\LEQ \a\LEQ1$.
Moreover, one has the :
\begin{lemma}\label{l5}
	Let $V\GEQ0$, locally integrable, $\vec U \in L^\infty (\Omega)$ be such that \eqref{eq:incompressible strong} and $0\LEQ\a<1$. Then, for all $\lambda>0$ and $f\in L^1(\O;\d^\a)$, there exists a unique function $\g\in D(A)$ such that 
	$$\g+\lambda A\,\g=f.$$
\end{lemma}
\begin{proof}
Indeed, since $L^1(\O;\d^\a)\subset L^1\big(\O;\d(1+|\Log\d|)\big)$, we may apply Theorem \ref{thm:well posedness in L1 delta 1 + log delta} to derive that for all $\lambda>0$ all $f\in L^1(\O;\d^\a)$ we have a unique function $\g\in L^{n',\infty}(\O)$ with $\dfrac\g\d\in L^1(\O)$,\ $V\g\in L^1(\O;\d^\a)$ and for all $\phi\in W^2L^{n,1}(\O)\in H^1_0(\O)$,
\begin{equation}\label{eq46}
\int_\O f\phi dx=\int_\O\g\big[\phi+\lambda(-\Delta\phi-\vec U\cdot\nabla \phi +V\phi)\Big]dx.
\end{equation}
This is equivalent to say that $\g+\lambda A\,\g=f$ and $\g\in D(A)$.
\end{proof}

So for $0\LEQ\a<1$, it remains to show that for all $\OV\g\in D(A),$ for all $\lambda>0$
\begin{equation}\label{eq47}
\|\OV\g\|_\a\LEQ\|\OV\g+\lambda A\,\OV\g\|_\a.
\end{equation}
That is to say, setting $f=\OV\g+\lambda A\,\OV\g$,
\begin{equation}\label{eq48}
\int_\O|\OV\g|\psi_1^\a dx\LEQ\int_\O|f|\psi_1^\a dx.
\end{equation}
To prove such inequality, we introduce as in \cite{RakoADE} the 
\begin{lemma}\label{l6}
	Let $\eps>0$, $0 \le \alpha \le 1$ and let
	\begin{equation}
		\psi_{1\eps}=(\psi+\eps)^\a-\eps^\a\in W^2L^{n,1}(\O)\cap H^1_0(\O).
	\end{equation}
	Then, for all $\OV\g\in L^{n',\infty}(\O), \ \OV\g\GEQ0$, one has
	\begin{equation}\label{eq49}
	J_\eps=\int_\O\OV\g\left[-\Delta \psi_{1\eps}-\vec U\cdot\nabla \psi_{1\eps}\right]dx\GEQ0.
	\end{equation}
\end{lemma}
\begin{proof}
We develop the term $-\Delta\psi_{1\eps}-\vec U\cdot\psi_{1\eps}$ to derive the
\begin{align*}
	J_\eps&=\a\int_\O\OV\g\Big[-\Delta\psi_1-\vec U\cdot\nabla\psi_1\Big](\psi_1+\eps)^{\a-1}dx+\a(1-\a)\int_\O|\nabla\psi_1|^2(\psi_1+\eps)^{\a-2}\OV\g dx\\
	&=\a\lambda_1\int_\O\OV\g\psi_1(\psi_1+\eps)^{\a-1}dx+\a(1-\a)\int_\O|\nabla\psi_1|^2(\psi_1+\eps)^{\a-2}\OV\g dx\GEQ0.
\end{align*}
\end{proof}
Let us decompose $f=f_+-f_-$, $f_+\in L^1(\O;\d^\a),\ f_-\in L^1(\O;\d^\a)$. By Theorem \ref{thm:well posedness in L1 delta 1 + log delta}, we know that we have $\g_1\in D(A)$ (resp. $\g_2\in D(A)$ such that
\begin{equation}\label{eq50}
\g_1+\lambda A\,\g_1=f_+\qquad\g_2+\lambda A\,\g_2=f_-.
\end{equation}
So by linearity and uniqueness, one has
\begin{equation}\label{eq51}
\OV\g=\g_1-\g_2.
\end{equation}
Therefore, it suffices to show that the inequality (\ref{eq48}) holds for $\g_1$ (resp. $\g_1$). That is to say that is sufficient to prove the inequality for $f\GEQ0$. But in that case, the unique solution of (\ref{eq46}) is non negative : $\OV\g\GEQ0$ and we can choose as a test function $\phi=\psi_{1\eps}$ given in Lemma \ref{l6}. We then have
\begin{equation}\label{eq52}
\int_\O f\psi_{1\eps}dx=\int_\O\OV\g\psi_{1\eps}dx+\lambda\int_\O\OV\g\Big[-\Delta\psi_{1\eps}-\vec U\cdot \nabla \psi_{1\eps}]dx+\lambda\int_\O V\,\psi_{1\eps}\OV\g dx.
\end{equation}
According to Lemma \ref{l6} and the fact that $V\,\g\,\psi_{1\eps}\GEQ0$ the two last integrals in relation (\ref{eq52}) are non negative. Therefore,
\begin{equation}\label{eq53}
\int_\O f\psi_{1\eps} dx\GEQ\int_\O\OV\g\,\psi_{1\eps}dx,\qquad\eps>0.
\end{equation}
Letting $\eps\to0$ in (\ref{eq53}), we obtain
\begin{equation}\label{eq54}
	\int_\O\OV\g\,\psi_1^\a dx\LEQ \int_\O f\psi_1^\a dx
\end{equation}
whenever $f\in\OV\g+\lambda A\,\OV\g,\ \OV\g\in D(A)$.\\

We have shown that the Schoedinger operator $A=-\Delta +\vec U\cdot \nabla +V$ is $m$-accretive in $L^1(\O,\d^\a)$, whenever $0\LEQ\a<1$, as in the first statement of  Theorem \ref{t32}.\qed\\

We have a similar result in $L^1(\O;\d)$ provided that $V(x)\GEQ c\d(x)^{-2}$ in a neighborhood $U$ of the boundary. The argument is similar to the preceding one but we need to replace the use of Theorem \ref{thm:well posedness in L1 delta 1 + log delta} by Theorem  \ref{t31}. Indeed, if $f=f_+-f_-\in L^1(\O;\d)$ and $\OV\g\in D(A)$ satisfies $\OV\g+\lambda A\,\OV\g=f$ then, Theorem \ref{t31} allows us to spleet $\OV\g=\g_2-\g_1$ with $\g_i\in D(A)$ and $\g_1+\lambda A\,\g_1=f_+$ (idem $\g_2+\lambda A\,\g_2=f_-$). therefore, it suffices to show the inequality
$$\int_\O\OV\g\,\psi_1dx\LEQ\int_\O f\psi_1dx\hbox{ for }f\GEQ0,\ \OV\g\GEQ0.$$

To do so, we choose $\phi=\psi_1$ in equation  (\ref{eq46}) and derive
\begin{equation}\label{eq55}
\int_\O f\psi_1dx=(1+\lambda\lambda_1)\int_\O\OV\g\psi_1 dx+\int_\O V\OV\g\,\psi_1dx.
\end{equation}
We drop the nonnegative term with $V$ to derive
\begin{equation}\label{eq56}
\int_\O\g\psi_1dx\LEQ\frac1{1+\lambda\lambda_1}\int_\O f\psi_1dx\LEQ\int_\O f\psi_1dx.
\end{equation}
This show the desired inequality and implies that \\$\forall\,\lambda>0,\ \forall\,\OV\g\in D(A),\ \OV\g+\lambda A\,\OV\g=f\in L^1(\O;d)$
\begin{equation}\label{eq57}
\int_\O|\OV\g|\psi_1 dx\LEQ\int_\O|\OV\g+\lambda A\,\OV\g|\psi_1dx.
\end{equation}
\ \HF
Therefore, we have shown the following theorem :

\begin{theorem}\label{t32}
	Let $\vec U\in L^\infty(\O)^n$ such that \eqref{eq:incompressible strong} and $V\GEQ0$ locally integrable. Then the Schrödinger operator 
	$$
		Au=-\Delta u+\vec U\cdot\nabla u+V u, \qquad \textrm{ for } u \in D(A)
	$$
	is $m$-accretive in $L^1(\O;\d^\a)$ for any $0\LEQ\a<1.$
	If $\a=1$ and  $V(x)\GEQ c\d(x)^{-2}$ in a neighborhood $U$ of the boundary then the operator $A$ is still $m$-accretive in $L^1(\O;\d)$.
\end{theorem}

The operator $A$ is also $m$-accretive in $L^p(\O; \d^ \alpha)$ when $\vec U = 0$ for $p \in (1,+\infty]$ and $\alpha \in [0,1]$. The result for the case $\alpha = 0$ was already proved by Brezis and Strauss \cite{Brezis-Strauss} for bounded potentials.

\begin{theorem} \label{thm:8.4}
	Let $p \in (1, + \infty]$. Assume that 
	\begin{equation} \label{eq:H star}
		\begin{dcases} 
			\alpha \in [0,1], \\
			\textrm{ and } \vec U = 0 
		\end{dcases}
		\qquad \textrm{ or } \qquad 
		\begin{dcases}
			\alpha = 0 , \\
			\textrm{ and } \eqref{eq:H with p greater n}.
		\end{dcases}
	\end{equation}
	Let $f \in L^p (\Omega, \delta^\alpha)$,  $0 \le V \in L^1 _{loc} (\Omega)$ and let $u \in D(A)$ be the unique solution of the equation 
	\begin{equation} \label{eq:+}
		A u + u = f .
	\end{equation}
	Then
	\begin{equation}
		\| u \|_{L^ p (\Omega; \delta^ \alpha)} \le \| f \|_{L^ p (\Omega; \delta^ \alpha)}.
	\end{equation}
\end{theorem}

\begin{proof}
	As in the proof of \Cref{t32} we can assume without loss of generality that $f \ge 0$ and thus $u \ge 0$. By regularity arguments it can be well-justified that we can take as test function the one given $u^{p-1} \psi_{1\varepsilon} (x)$ with $\psi_{1,\varepsilon}$ as in \Cref{l6} if $\vec U =0$ and $u^{p-1}$ if $\alpha = 0$. Then, from \eqref{eq:vws Brezis}, and since $ V \ge 0 $, we get that
	\begin{equation*}
		\int_ \Omega |u|^p \psi_{1,\varepsilon} + I \le \int_ \Omega f u^ {p-1} \psi_{1\varepsilon}\le \left( \int _ \Omega f^ p \psi _{1 \varepsilon} \right)^{\frac 1 p} \left(  \int_{ \Omega} u^p \psi_{1\varepsilon}  \right)^ {  \frac {p-1} p },
	\end{equation*}
	where
	\begin{equation*}
		I = \int_ \Omega u (- \Delta(u^{p-1} \psi_ {1 \varepsilon})) - \int_{ \Omega } u\ \vec U \cdot \nabla (u^{p-1} \psi_{1,\varepsilon}).
	\end{equation*}
	We recall that in $L^p (\Omega, \delta^ \alpha)$ we can use as an equivalent norm to the one given by
	\begin{equation*}
		\left( \int _ \Omega |u|^p \psi _{1 \varepsilon} \right) ^ {\frac 1 p} .
	\end{equation*}
	Thus, it is enough to prove that $I \ge 0$. Assume now that $\vec U = 0$. Again, we can assume that $u$ is regular and so
	\begin{equation*}
		I = - \int _ \Omega \Delta u (u^ {p-1} \psi _{1,\varepsilon})=(p-1) \int _ \Omega u^ {p-2} |\nabla u|^2 \psi _{1,\varepsilon} + \int _ \Omega u^ {p-1} \nabla u \cdot \nabla \psi_{1,\varepsilon}.
	\end{equation*}
	The first integral is clearly nonnegative. Moreover
	\begin{equation*}
		\int _ \Omega u^ {p-1} \nabla u \cdot \nabla \psi_{1,\varepsilon} = \int _ \Omega \frac 1 p \nabla u^ p \cdot \nabla \psi_{ 1, \varepsilon} = \int _ \Omega \frac {u^p} {p} (-\Delta \psi _{1,\varepsilon}),
	\end{equation*}
	and, from the definition of $\psi_{1,\varepsilon}$, we get that $I \ge 0$. This concludes the proof for the case $\vec U = 0$.
	\\
	Assume now that $\alpha = 0$. Then, by applying Lemma 2.6 in \cite{DGRT}, we get that 
	$$\int_ \Omega u \ \vec U \cdot \nabla u^{ p -1 } = 0$$
	and again $I \ge 0$.
\end{proof}
As a first application of \Cref{t32} and \Cref{thm:8.4} we get the solvability of the associated parabolic problem
\begin{equation}\label{eq:PP}
	\begin{dcases}
		\frac{\partial u}{\partial t} - \Delta u + \vec U \cdot \nabla u + V (x) u = f(x,t) & \textrm{in } \Omega \times (0,T) , \\
		u(x,0) = u_0 (x) &\textrm{on } \partial \Omega,
	\end{dcases}
\end{equation}
for the class $0 \le V \in L^1_{loc} (\Omega)$ and thus also for very singular potentials. Here is a simple statement in term of ``mild solutions'' (see, e.g. \cite{Barbu}, \cite{Crandall-Liggett} or Proposition 1.5.14 in \cite{Brezis-Cazenave}).

\begin{theorem} \label{thm:8.4 2}
	Let $T > 0$, $\alpha \in [0,1]$, $\vec U\in L^\infty(\O)^n$ such that \eqref{eq:incompressible strong} and $V\GEQ0$ locally integrable (satisfying \eqref{singular potential} if $\alpha = 1$). Let $u_0 \in L^1 (\Omega; \delta^ \alpha)$ and $f \in L^1 (0,T; L^1 (\Omega; \delta^ \alpha))$. Then, there exists a unique $u \in C([0,T]: L^1 (\Omega; \delta^ \alpha))$ mild solution of \eqref{eq:PP}. Moreover $V (x) u \in L^1 (0,T; L^1 (\Omega, \delta^ \alpha))$ and, if $\hat u$ denotes the solution to data $\hat u_0$ and $\hat f$ under the same assumptions, then, for any $t \in [0,T]$,
	\begin{equation} \label{eq:**}
		\| ( u(t, \cdot) - \hat u(t, \cdot) )_+ \|_{L^1 (\Omega; \delta^ \alpha)} \le \| (u_0 - \hat u_0)_+ \|_{L^1 (\Omega, \delta^ \alpha)} + \int_0 ^t \| ( f(t, \cdot) - \hat f( t, \cdot)  )_+ \|_{L^1 (\Omega, \delta^ \alpha)}.
	\end{equation}
	In addition, if $\vec U = \vec 0$, $u_0 \in L^p (\Omega, \delta^ \alpha)$ and $f \in L^1 (0,T; L^p (\Omega, \delta^ \alpha))$ for some $p \in (1, +\infty] $, then $u \in C ([0,T]; L^p (\Omega, \delta^ \alpha))$ and \eqref{eq:**} holds replacing the norm of $L^1 (\Omega, \delta^ \alpha)$ by the norm of $L^p (\Omega, \delta^ \alpha)$.
\end{theorem}

The application of abstract semigroup theory results on the time differentiability of solutions of \eqref{eq:PP} requires the reflexivity condition on the abstract Banach space. This holds in the case of the second part of \Cref{thm:8.4 2} when $1 < p < +\infty$ (and $\vec U = 0$ or  $\alpha = 0$). Nevertheless, a direct approach to this question for problem \eqref{eq:PP} can be obtained as an application of Proposition 1.3.4 of \cite{Brezis-Cazenave} if $ f= 0$ and Proposition 1.5.5 if $f \ne 0$. We have 
\begin{theorem}\label{t50}
	Let $T>0,\ \g_0\in D(A)$, $f\in C([0,T]; D(A)) \cup C^1([0,T];L^1(\Omega;\d^\a))$ for some $\a\in[0,1]$. Then, there exists a (unique) function satisfying :$$
	\begin{cases}
	\g\in C\Big([0,T];D(A)\Big)\cap C^1\Big([0,T];L^1(\O;\d^\a)\Big)\\
	\dfrac{d\g}{dt}(t)+A\,\g(t)=f(t),\ \forall\,t\in[0,T],\quad \g(0)=\g_0.
	\end{cases}$$
\end{theorem}
\begin{remark}
	It is possible to obtain several qualitative properties of solutions of the parabolic problem \eqref{eq:PP}. The smoothening effect for bounded potentials can be found, e.g., in \cite{Brezis-Cazenave,Brezis-Cazenave-b,Fila-Souplet-Weissler,RakoADE}. If $V(x)$ is a very singular potential then the Dirichlet condition is verified in $W_0^1 L^{n', \infty} (\Omega)$ once we assume that $\alpha \in [0,1)$. In fact, it is not complicated to adapt the techniques of proof of the \Cref{subsec:8.2} of this paper to show that if $u_0$ and $f(t, \cdot)$ are ``flat'' data near $\partial \Omega$ then the (unique) solution of \eqref{eq:PP} is also a ``flat solution'' in the sense that not only $u = 0$ but $\frac {\partial u} {\partial \vec \nu} = 0$ on $\partial \Omega$. Notice that this is in contrast with the instantaneous blow-up of solutions which arises when  $\vec U = 0$, $V(x) < 0$, $\lambda_1 (-\Delta + (1-\eps ) V ) = -\infty$ for some $\eps > 0$ and $V$ is very singular (see \cite{Cabre-Martel} and the references therein).
\end{remark}

\subsection{Complex Schrödinger problem}
\label{subsec:8.2}
Let us apply our previous results to the mathematical treatment of problem %
\eqref{Schroedinger  whole space}. In some sense, our main aim now is to show
that the solution of this Schr\"{o}dinger equation is \emph{localized} for
any\emph{\ }$t>0$,\emph{\ }in the sense that if we start with a localized
initial wave packet $\bm \psi _{0}\in H^{1}(\mathbb{R}^{n}:\mathbb{C})$ (here $%
\mathbb{C}$ denotes the complex numbers), i.e. such that 
\begin{equation*}
\text{support }\bm {\psi }_{0}\subset \overline{\Omega },
\end{equation*}%
then the particle still remains permanently confined in $\Omega $ in the
sense that 
\begin{equation*}
\text{support}\bm {\psi }(t,\cdot )\subset \overline{\Omega }\text{ for
any }t>0.
\end{equation*}%
As in \cite{Diaz 2} we start by considering the auxiliary eigenvalue problem 
\begin{equation*}
DP(V,\lambda ,\Omega ) \quad \begin{dcases} -\Delta u +\vec U\cdot
\nabla u +V(x)u =\lambda u & \hbox{in $\Omega$,} \\ u =0
& \hbox{on $\partial \Omega$.}\end{dcases}
\end{equation*}

\begin{proposition}
	\label{Prop:Existence eigenvalues}
	Let $0 \le V \in L^1_{loc} (\Omega)$. Then
	there exists\textbf{\ }a sequence of eigenvalues $\lambda _{m}\rightarrow
	+\infty $, $\lambda _{1}>$ $\lambda _{1,\Omega }$ (the first eigenvalue for
	the Dirichlet problem associated to the operator $-\Delta +\vec U\cdot
	\nabla $ on $\Omega $), $\lambda _{1}$ is isolated and $u _{1}>0$ on $%
	\Omega .$
\end{proposition}

\begin{proof}  
	We start by arguing as in the proof of Proposition 3 of \cite{DGRT}.
	We introduce the space
	\begin{equation*}
	W=\left\{ \varphi \in H_{0}^{1}(\Omega ):V\varphi ^{2}\in L^{1}(\Omega
	)\right\} .
	\end{equation*}%
	For any $h\in L^{2}(\Omega )$ we define the operator $Th=z\in W$ solution of
	the linear problem 
	\begin{equation}
	\begin{dcases} Az=h\text{ } & \text{in }\Omega , \\ z=0 & \text{on }\partial
	\Omega .\end{dcases}  \label{problem h}
	\end{equation}%
	We recall that the existence and uniqueness of a solution was obtained in
	Proposition 3 in \cite{DGRT} when $h\in W^{\prime }$ (the dual space of $W)$ and that,
	trivially, $L^{2}(\Omega )$ $\subset W^{\prime }.$Then the composition with
	the (compact) embedding $H_{0}^{1}(\Omega )\subset L^{2}(\Omega )$ is a
	selfadjoint compact linear operator $\widetilde{T}=i\circ T:L^{2}(\Omega
	)\rightarrow L^{2}(\Omega )$ for which we obtain in the usual way a sequence
	of eigenvalues $\lambda _{m}\rightarrow +\infty $. By well-known results
	(see e.g. \cite{Red-Simon 2} or \cite{Brezis libro}) we know that $\lambda
	_{1}>0$ (notice that $\lambda _{1}=0$ would imply that $z=0$)$.$ In fact,
	since $V(x)\geq 0$, by the comparison principle we know that $\lambda
	_{1}>\lambda _{1,\Omega }$. The positivity of the first eigenfunction $%
	u _{1}$ is an easy modification of Proposition 3.2 of \cite{Diaz 2}.
	Moreover a variant of the Krein-Rutman can be applied (see \cite{Daners})
	and so we know that $\lambda _{1}$ is isolated.
\end{proof}

\begin{remark}
When $r=2$ in (\ref{singular potential}) then, by the Hardy inequality, $%
W=H_{0}^{1}(\Omega )$.
\end{remark}

A different, and useful, consequence of Proposition 3 of \cite{DGRT} is the
following:

\begin{proposition}
	\label{Mononicity in L2}
	Let $0 \le V \in L^1_{loc} (\Omega)$. Then the operator
	$A:D(A)(\subset L^{2}(\Omega ))\rightarrow L^{2}(\Omega )$ given by 
	$D(A)=W=\left\{ \varphi \in H_{0}^{1}(\Omega ):V\varphi ^{2}\in L^{1}(\Omega
	)\right\} $ and $Au =-\Delta u +\vec U\cdot \nabla u
	+Vu $ if $u \in D(A)$ is a maximal monotone operator in $%
	L^{2}(\Omega )$.
\end{proposition}

\begin{proof} 
	 Given $h\in L^{2}(\Omega )$, the existence and uniqueness of solution of the equation $Au +u =h$ was obtained in Proposition 3 of \cite{DGRT}. Moreover, thanks to the assumptions on $\vec U$, by Lemmas 2.6 and
	2.7 of \cite{DGRT} we get that
	\begin{equation*}
	\left\Vert u \right\Vert _{L^{2}(\Omega )}\leq \left\Vert h\right\Vert
	_{L^{2}(\Omega )}
	\end{equation*}%
	which proves the monotonicity in $L^{2}(\Omega )$ (i.e. the operator is
	m-accretive in $L^{2}(\Omega )$).
\end{proof}

Let us prove now that the singularity of the potential implies that
all the eigenfunctions $u _{m}$ of operator $A$ are flat solutions (in
the sense that $u =\frac{\partial u }{\partial n}=0$ on $\partial
\Omega $). As usual in Quantum Mechanics we shall pay attention to the
associate eigenfunctions with normalized $L^{2}$-norm, i.e. such that 
\begin{equation}
\left\Vert u _{m}\right\Vert _{L^{2}(\Omega )}=1.
\label{eq:normalization}
\end{equation}

\begin{theorem}
\label{Thm: energy flat}
	Assume \eqref{singular potential} and let $u _{m}$ be an eigenfunction
	associated to the eigenvalue $\lambda _{m}$. Then $u _{m}\in L^{\infty
	}(\Omega )$ and $u _{m}$ is a flat solution of $DP(V,\lambda
	_{m},\Omega )$. In fact, there exists $\overline{K}_{m}>0$ such that 
	\begin{equation}
	\left\vert u _{m}(x)\right\vert \leq \overline{K}_{m}d(x,\partial
	\Omega )^{2}\quad \text{a.e. }x\in \Omega .  \label{eq:flat exponent 2}
	\end{equation}
\end{theorem}

\begin{proof}  
	It suffices to repeat all the arguments of Theorem 2.1 of \cite{Diaz 2} (concerning the case $r=2$ and $\vec U=\vec {0}$) since the
	the main idea of the proof consists in the use of a Moser-type iterative
	argument (as in \cite{Drabek-Hernandez}) and take as test functions
	\begin{equation}
	\varphi (x)=v_{m,M}^{2\kappa +1}(x)\text{, with }v_{m,M}^{{}}(x):=\min
	\{\left\vert u _{m}(x)\right\vert ,M\}\sign(u _{m}(x)),
	\label{Test truncated}
	\end{equation}%
	for any arbitrary $M,\kappa >0$. Then, by using (\ref{singular potential})
	and Lemmas 2.6 and 2.7 of \cite{DGRT}\ we conclude that $\varphi \in
	H_{0}^{1}(\Omega )$ is an appropriate test function and
	
	\begin{align}
	(2\kappa +1)\int_{\Omega }\left\vert v_{M}^{2\kappa }(x)\right\vert &
	\left\vert \nabla u _{m}\right\vert ^{2}dx+\int_{\Omega }\frac{%
	\underline{C}}{\delta (x)^{2}}\left\vert v_{M}^{2\kappa +1}(x)\right\vert
	\left\vert u _{m}\right\vert dx  \nonumber \\ 
	&\leq (2\kappa +1)\int_{\Omega }\left\vert v_{M}^{2\kappa }(x)\right\vert
	\left\vert \nabla u _{m}\right\vert ^{2}dx+\int_{\Omega }V(x)\left\vert
	v_{M}^{2\kappa +1}(x)\right\vert \left\vert u _{m}\right\vert dx \nonumber \\ 
	&=\lambda _{n}\int_{\Omega }\left\vert v_{M}^{2\kappa +1}(x)\right\vert
	\left\vert u _{m}\right\vert dx
	\label{eq: energy estimate}
	\end{align}%
	where we used the simplified notation $v_{M}=v_{m,M}$. This is exactly the
	same starting energy estimate than the one used in the proof of Theorem 2.1
	of \cite{Diaz 2} and thus the rest of the proof (passing to the limit when $%
	M\nearrow +\infty $) applies without any other modification.
\end{proof}

\begin{remark}
The flatness of the eigenfunctions $u _{m}$ of operator $A$ can be also
proved by using Proposition 2.7 of \cite{OrPon} nevertheless the statement
given here supplies some decay estimates on $u _{m}$ near $\partial
\Omega $ which are not given in the mentioned reference.
\end{remark}

\begin{remark}
The decay estimate (\ref{eq:flat exponent 2}) is not optimal if $r>2$ in (%
\ref{singular potential}). It seems possible to adapt the formal exposition
made in \cite{Frank} developing asymptotically some Bessel functions to
prove that in that case 
\begin{equation}
\left\vert u _{m}(x)\right\vert \leq \overline{K}_{m}{\delta(x)}^{r/4}\exp \left (-\frac{\widehat{K}_{m}}{(r-2)}\delta(x)^{-(r-2)/2} \right)\quad \text{a.e. }x\in \Omega ,
\label{exponential negative estimate}
\end{equation}%
for some positive constants $\overline{K}_{m}$ \ and $\widehat{K}_{m},$\ but
we shall not enter into the details here.
\end{remark}

\begin{remark}
Arguing as in \cite{Diaz 2} it is easy to get several qualitative properties
of solutions of the complex evolution Schrödinger problem

\begin{equation}  \label{Problem Evolution V}
\begin{dcases} 
	\mathbf{i}\frac{\partial \bm \psi}{\partial t}=-\Delta
		\bm \psi+\vec U\cdot \nabla \bm \psi+V(x)\bm \psi &
\text{in }(0,\infty )\times \mathbb R^{n} \\ \bm \psi
(0,x)=\bm \psi_{0}(x) & \text{on }\mathbb R^{n}\end{dcases}
\end{equation}%
for very singular potentials over $\Omega $ which are extended (for
instance) in a finite way to the whole space. So, we assume now that there
exists $q\in \lbrack 0,+\infty )$ such that%
\begin{equation}
V_{q,\Omega }(x)= 
\begin{cases}
V(x) & \text{if }x\in \Omega , \\ 
q & \text{if }x\in \mathbb R^{n}\setminus\Omega%
\end{cases}
\label{Potential extension q}
\end{equation}
and that (\ref{singular potential}) holds. We can study the time evolution
of a localized initial wave packet $\psi _{0}\in H^{1}(\mathbb R^{n}:\mathbb C)$ such that
support $\psi _{0}\subset \overline{\Omega }.$

Then we can prove that there exists a unique solution $\bm \psi \in
C([0,+\infty ):L^{2}(\mathbb R^{n}:\mathbb C))$ with 
$\bm \psi%
\in C([0,+\infty ):H^{1}(\mathbb{R}^{n}:\mathbb{C}))$
 and $V_{q,\Omega }(x)%
\mathbf{\psi \in }L^{2}(0,T:L^{2}(\mathbb{R}^{n}:\mathbb{C}))\}$
for any $T>0,$ and that the Galerkin decomposition 
\begin{equation}
\bm \psi_{\Omega }(t,x)=\sum_{m=1}^{\infty }\mathbf{a}_{m}e^{-\mathbf{i%
}\lambda _{m}t}u _{m}(x),  \label{Galerkin decomposition}
\end{equation}%
holds with convergence at least in $L^{2}(\mathbb{R}^{n}:\mathbb {C})$ where $\lambda _{m}$ and $u _{m}$ are the eigenvalues and
eigenfunctions given in Proposition \ref{Prop:Existence eigenvalues} and 
\begin{equation*}
\mathbf{a}_{m}=\int_{\Omega }\bm \psi_{0}(x)u _{m}(x)dx.
\end{equation*}%
For localizing purposes we assume that 
\begin{equation}
\sum_{m=1}^{\infty }\left\vert \mathbf{a}_{m}\right\vert \overline{K}%
_{m}<+\infty ,  \label{Hipotesis dato inicial}
\end{equation}%
where $\overline{K}_{m}>0$ was given in Theorem \ref{Thm: energy flat}.
Thus, we conclude that%
\begin{equation}
\left\vert \bm \psi(t,x)\right\vert \leq Kd(x,\partial \Omega
)^{2}\quad \text{for any }t>0\text{ and a.e. }x\in \Omega ,
\label{Estimate psi omega}
\end{equation}%
for some $K>0,$ and in consequence the unique solution of (\ref{Problem
Evolution V}) satisfies that support $\bm \psi(t,.)\subset \overline{%
\Omega }$ for any $t>0$.

Concerning the existence of solutions, it is enough to apply the
Hille-Yosida theorem (see, e.g. \cite{Red-Simon 2,Brezis libro,Brezis-Cazenave}).
For the Galerkin decomposition we can adapt the arguments given in \cite%
{Brezis-Cazenave}.
\end{remark}

\section*{Acknowledgments}

The authors thank Haïm Brezis for fruitful discussions. The idea behind the improved approximation result, \Cref{t3} comes from a remark of Haïm Brezis to the second author (D. G.-C.). The research of D. G\'omez-Castro was supported by a FPU fellowship from the Spanish government. The research of J.I. D\'iaz and D. G\'omez-Castro was partially supported by the project ref. MTM 2014-57113-P of the DGISPI (Spain).


\begin{thebibliography}{99}
	
\bibitem{Barbu}{V. Barbu, Nonlinear Differential Equations of Monotone Types in Banach Spaces, Springer New York, New York, NY, 2010. doi:10.1007/978-1-4419-5542-5.}
	
\bibitem{Bennet-Sharpley}{C. Bennet, R. Sharpley, \emph{Interpolation of Operators}. Academic Press, Boston (1988)}	

\bibitem{BBGGPV} { Ph. B\'enilan, L. Boccardo, Th. Gallou\"et, R. Gariepy, M. Pierre, J.L. Vázquez,} An $L^1$-theory of existence and uniqueness of solutions of nonlinear elliptic equations, {\it Ann. Scuala Norms Sup. Pisa}, {\bf 22} (1995) 241-273.
	
\bibitem{Brezis libro} H. Brezis, \textit{Functional Analysis, Sobolev
Spaces and Partial Differential Equations, }Springer, New York, 2011.

\bibitem{Brezis-Cazenave} H. Brezis and T. Cazenave, \textit{Linear
semigroups of contractions: the Hille-Yosida theory and some applications}.
Publications du Laboratoire d'Analyse Num\'{e}rique, Universit\'{e} Pierre
et Marie Curie, Paris, 1993.

\bibitem{Brezis-Cazenave-b} H. Brezis and T. Cazenave.: A Nonlinear Heat Equation with Singular Initial Data. J. d’Analyse Math{é}matique. 68, 277–304 (1996)

\bibitem{Brezis-Strauss}{H. Brézis and W.A. Strauss, Semi-linear second-order elliptic equations in $L^1$, J. Math. Soc. Japan. 25 (1973) 565–590.}

\bibitem{Cabre-Martel}{X. Cabré and Y. Martel, Existence versus explosion instantanée pour des équations de la chaleur linéaires avec potentiel singulier, \emph{Comptes Rendus l’Académie des Sci. - Ser. I - Math.} 329 (1999) 973–978. doi:10.1016/S0764-4442(00)88588-2.}

\bibitem{Crandall-Liggett}{M.G. Crandall and T.M. Liggett, Generation of Semi-Groups of Nonlinear Transformations on General Banach Spaces, \emph{American Journal of Mathematics}. 93 No.2(1971) 265–298.}

\bibitem{Fila-Souplet-Weissler}{M. Fila, P. Souplet and F.B. Weissler, Linear and nonlinear heat equations in $L_\delta^q$ spaces and universal bounds for global solutions. \emph{Mathematische Annalen} 113, 87–113 (2001). doi:10.1007/s002080100186}

\bibitem{Frank} W.M. Frank and D.J. Land, Singular potentials. Rev. Mod.
Phys. 43(1) (1971) 36--98.

\bibitem{Daners} D. Daners and P. Koch Medina, Abstract evolution equations,
periodic problems and applications, Longman, Harlow, 1992.

\bibitem{Diaz1} J.I. D\'{\i}az, On the ambiguous treatment of Schrödinger
equations for the infinite potential welland an alternative via flat
solutions: The one-dimensional case. \emph{Interfaces and Free Boundaries} 17 \textbf{3} 
(2015) 333-351.

\bibitem{Diaz 2} J.I. D\'{\i}az, On the ambiguous treatment of the
Schrödinger equation for the infinite potentialwell and an alternative via
singular potentials: the multi-dimensional case. \emph{SeMA-Journal} 74 \textbf{3} (2017) 225-278, DOI
10.1007/s40324-017-0115-3,.

\bibitem{DGRT} J. I. D\'{\i}az, D. G\'{o}mez-Castro, J.M. Rakotoson and
R. Temam, Linear diffusion with singular absorption potential and/or
unbounded convective flow: the weighted space approach. To appear in \emph{Discrete and
Continuous Dynamical Systems}.

\bibitem{DR} {J.I. Díaz and J.-M. Rakotoson, Elliptic Problems on the Space of Weighted With the Distance To the Boundary Integrable Functions Revisited. Electron. J. Differ. Equations Conf. 21, 45–59 (2012)}

\bibitem{Drabek-Hernandez} P. Dr\'{a}bek and J. Hern\'{a}ndez, {\normalsize %
Quasilinear eigenvalue problems with singular weights for the p-Laplacian.
To appear. }

\bibitem{g1}{D. Gilbarg and S. Trudinger,} {\it Elliptic partial differential equations of second order}, Springer, Berlin 2001.

\bibitem{Gomez-Castro} D. G\'{o}mez-Castro, Homogenization and Shape
Differentiation of Quasilinear Elliptic Problems. Applications to
Chemical Engineering and Nanotechnology. Thesis at the UCM. 2017.

\bibitem{Ionescu-Kening} {A.D. Ionescu and C. E. Kenig,} {Uniqueness properties
of solutions of Schr\"{o}dinger equations}, J. Funct. Anal. 232 (2006), no.
1, 90--136., https://doi.org/10.1016/j.jfa.2005.06.005


\bibitem{PazyBook}{A. Pazy}, {\it Semigroups of Linear Operators and Applications Differential Equations}, Applied Math. Sciences N44, Springer, New-York, 1983    

\bibitem{RakoBook}{  J.M. Rakotoson,} {\it Linear equations with variable coefficients and Banach function spaces} To appear.


\bibitem{RakoADE}{J.M. Rakotoson,} {Regularity of a very weak solution for parabolic equations and applications} {\it   Advances in Differential Equations}  {\bf 16} (2011)  867-894.


\bibitem{RakoJFA}{J.M. Rakotoson,} {New Hardy inequalities and behaviour of linear elliptic equations} {\it   Journal of Functional Analysis}  {\bf 263} (2012)  2893-2920.

\bibitem{Red-Simon 2} M. Reed and B. Simon, \textit{Methods of Modern
Mathematical Physics, Vol. II}, Academic Press, New York 1975.

\bibitem{OrPon} L. Orsina and A. Ponce, Hopf potentials for the Schrödinger
operator (version of 12 April 2017).

	
	
	

	%

	
	
\end{thebibliography}
\end{document}